\documentclass[11pt]{article}
\let\endproof\relax

\usepackage{fullpage}
\usepackage{amssymb}
\usepackage{amsmath}
\usepackage{latexsym}
\usepackage{graphicx}
\usepackage{color}
\usepackage{url}

\usepackage{Shorthands}

\usepackage{listings}
\usepackage{courier}
\usepackage{comment}
\usepackage[ruled]{algorithm2e}
\SetKwInput{KwInput}{Input}                % Set the Input
\SetKwInput{KwOutput}{Output}              % set the Output

\theoremstyle{definition}

\theoremstyle{definition}

\theoremstyle{remark}

\renewcommand{\hat}{\widehat}

\usepackage{mathtools}

\newcommand{\thickbar}[1]{\bm\bar{#1}}

\newcommand{\hattheta}{\ensuremath{\hat{\theta}}}

\newcommand{\hatSigmaT}{\ensuremath{\hat{\Sigma}_T}}
\newcommand{\barSigma}{\ensuremath{\thickbar{\Sigma}}}
\newcommand{\barX}{\ensuremath{\thickbar{X}}}
\newcommand{\starG}{\ensuremath{\calG_{\star}}}

\newcommand{\dx}{d_{\mathsf{x}}}

\newcommand{\dtheta}{d_{\mathsf{\theta}}}
\DeclareMathOperator{\gradtheta}{\nabla_{\theta}}
\DeclareMathOperator{\Btheta}{B_{\theta}}
\renewcommand{\P}{\ensuremath{\mathbf{P}}}
\newcommand{\E}{\ensuremath{\mathbf{E}}}
\usepackage[]{footmisc}
\renewcommand{\norm}[1]{\ensuremath{\| #1 \|}}
\newcommand{\normbig}[1]{\ensuremath{\bigg\| #1 \bigg\|}}
\DeclareMathOperator{\Gammadep}{\Gamma_{\textup{dep}}}
\newcommand{\sumt}{\ensuremath{\sum_{t=0}^{T-1}}}
\usepackage{chngcntr}
\usepackage{apptools}
\AtAppendix{\counterwithin{lemma}{section}}
\AtAppendix{\counterwithin{definition}{section}}
\newcommand{\pars}[1]{\ensuremath{\left( #1 \right)}}
\newcommand{\bracks}[1]{\ensuremath{\left[ #1 \right]}}
\DeclareMathOperator{\T}{\intercal}

\newcommand{\sfX}{\mathsf{X}}

\newcommand{\sfZ}{\mathsf{Z}}
\newcommand{\sfP}{\mathsf{P}}

\newcommand{\sfM}{\mathsf{M}}

\newcommand{\e}{\varepsilon}
\usepackage[page]{appendix}
\usepackage{etoolbox}
\patchcmd{\endproof}
  {\popQED}
  {\par\popQED}
  {}
  {}

\newcommand{\normGamma}{\ensuremath{\norm{\Gamma_{\textup{dep}}(\sfP_{\sfX})}}}

\author{Charis Stamouli\thanks{The authors are with the Department of Electrical and Systems Engineering, University of Pennsylvania, Philadelphia, PA 19104, USA. Emails: \texttt{\{stamouli,ingvarz,pappasg\}@seas.upenn.edu}.} , Ingvar Ziemann\footnotemark[1] , and George J. Pappas\footnotemark[1] }

\date{}

\usepackage[utf8]{inputenc} % allow utf-8 input
\usepackage[T1]{fontenc} %\usepackage{ae} \usepackage{aecompl}   % use 8-bit T1 fonts

\usepackage{fullpage}

\title{Rate-Optimal Non-Asymptotics for the \\ Quadratic Prediction Error Method}

\usepackage{Shorthands}
\begin{document}

\maketitle

\begin{abstract}
We study the quadratic prediction error method---i.e., nonlinear least squares---for a class of time-varying parametric predictor models satisfying a certain identifiability condition. While this method is known to asymptotically achieve the optimal rate
%rate-optimal decay of the prediction error 
for a wide range of problems, there have been no \emph{non-asymptotic} results matching these optimal rates outside of a select few, typically linear, model classes.
%known to provide parameter estimates that ensure rate-optimal asymptotic convergence of the true prediction error to zero, its finite-sample performance is yet to be characterized, except for a few special---typically linear---model classes. 
By leveraging modern tools from learning with dependent data, we provide the first rate-optimal non-asymptotic analysis of this method for our more general setting of nonlinearly parametrized model classes. Moreover, we show that our results can be applied to a particular class of \textit{identifiable} AutoRegressive Moving Average (ARMA) models, resulting in the first optimal non-asymptotic rates for identification of ARMA models. 
\end{abstract}

\section{Introduction}
Identifying predictive models from data is of critical importance in a wide range of fields, from classical control theory and signal processing to modern machine learning. To this end, a significant line of work in system identification has been devoted to identifying predictor models of the form:
\begin{equation}\label{model}
    Y_t=f_t(X_t,\theta_{\star})+W_t,
\end{equation}
from sequential data $(X_0,Y_0),\ldots,(X_{T-1},Y_{T-1})$.  We typically refer to the variables $X_t$ as the inputs and the variables $Y_t$ as the outputs, with the inputs allowed to have a causal dependence on past outputs. However, we do not restrict attention to input-output models in the sense that \eqref{model} may well be autonomous, cf. \eqref{eq:arma} below. 

Assuming that the regression functions $f_t(\cdot,\cdot)$ are known, a standard approach for estimating the unknown parameter $\theta_{\star}$ is to minimize the quadratic criterion:
\begin{equation*}\label{eq:criterion}
    L_T(\theta):=  \frac{1}{T}\sum_{t=0}^{T-1} (f_t(X_t,\theta)-Y_t)^2
\end{equation*}
over a parameter class $\sfM$ which is assumed to contain $\theta_{\star}$. This approach yields the quadratic prediction error method, also referred to as nonlinear least squares. 

As a motivating example, consider the classical prediction error method for AutoRegressive Moving Average (ARMA) models of the form:
\begin{equation}\label{eq:arma}
    Y_t = \sum_{i=1}^p a_i^{\star} Y_{t-i}+ \sum_{j=0}^q b_j^{\star} W_{t-j}
\end{equation}
from system identification \cite{ljung1998system, Tsybakov2009}. Such models can be cast in the form \eqref{model} with parameter $\theta_\star:=\begin{bmatrix}a_1^{\star},\ldots,a_p^{\star},b_0^{\star},\ldots,b_q^{\star}\end{bmatrix}^{\T}$ and inputs $X_t:=\begin{bmatrix}Y_0,\ldots,Y_{t-1}\end{bmatrix}^{\T}$. To convert \eqref{eq:arma} to the form \eqref{model}, one selects $f_t(\cdot,\cdot)$ to be the conditional expectation of $Y_t$ given  all the past data $Y_0,\ldots,Y_{t-1}$. We return to this example in more detail in \Cref{sec:examples}.

While the asymptotic rates of prediction error methods are by now well understood---including optimal rates of convergence \cite{ljung1998system} as characterized by the Cram\'er-Rao Inequality---relatively less is known about their non-asymptotic counterparts. Some early progress on extending these ideas to the finite-sample regime was made in \cite{campi2002finite}. However, the bounds therein are both qualitatively and quantitatively loose as compared to older asymptotic results.  

A few years ago, drawing upon recent advances in high-dimensional statistics and probability \cite{Vershynin2018, Wainwright2019}, non-asymptotic rates nearly as sharp as the older known asymptotics were derived for the particular case of fully observed ARMA models, given by $Y_t = a_1^{\star} Y_{t-1}+W_t$ \cite{Simchowitz2018, faradonbeh2018finite}.   Soon thereafter, classical subspace methods from system identification, based on higher-order linear autoregressions \cite{jansson1998consistency,chiuso2004asymptotic,qin2006overview}, were also given a refined non-asymptotic analysis \cite{tsiamis2019finite}. Note that in contrast to the general prediction error method, the algorithms in \cite{Simchowitz2018, faradonbeh2018finite,tsiamis2019finite} are based on linear least squares. For a broader overview of recent results on non-asymptotic learning and identification of linear models, refer to \cite{tsiamis2023statistical, ziemann2023tutorial}.

As for learning and identification of nonlinear models of the form \eqref{model}, progress on non-asymptotic analysis has proven somewhat slower. The special case of a generalized linear model (i.e., a first-order linear autoregression composed with a static known nonlinearity) is analyzed in \cite{Sattar2022, Kowshik2021}. At a technical level, the goal has primarily been to sidestep---as much as possible---the blocking technique \cite{Yu1994}, which has otherwise been a dominant approach to deriving non-asymptotic guarantees for learning with dependent data \cite{Mohri2008,Duchi2012,Kuznetsov2017,Roy2021,Sancetta2020}. In brief, the blocking technique splits a dependent sample $\{Z_t\}_{t=0}^{T-1}$ into independent blocks, say $\{Z_t\}_{t=1}^k, \{Z_t\}_{k+1}^{2k},\ldots$, and then proceeds to treat each block as an independent datapoint. The caveat of this technique is that it reduces the effective sample size (e.g., here  by a factor of $k$) and thus typically does not yield optimal rates of convergence. To provide some intuition, $k$ above can be thought of as an analogue to the inverse stability margin of a linear system, and in fact, the blocking technique cannot be applied to marginally stable linear autoregressions. By contrast, an optimal asymptotic characterization of the rate of convergence for such autoregressions has been known since 1943 \cite{mann1943statistical}. Moreover, note that sidestepping this approach is precisely what allowed \cite{Simchowitz2018} to first derive optimal rates for linear system identification. 

More recently, \cite{Ziemann2022} showed how to, at least partially, avoid the blocking approach for the time-invariant version of \eqref{model}---with $f_t(\cdot,\cdot)=f(\cdot,\cdot)$ for a fixed function $f(\cdot,\cdot)$ independent of time. The result of \cite{Ziemann2022} is almost sufficient to provide a rate-optimal non-asymptotic analysis of the ARMA prediction error method. However, it has two shortcomings for this purpose, one of which we have already hinted at. First, the result does not allow for time-varying regression functions $f_t(\cdot,\cdot)$, which is crucial, as the conditional expectation function of $Y_t$ given the past data $Y_0,\ldots,Y_{t-1}$ generally varies in time. Second, the final bounds in \cite{Ziemann2022} are loose by logarithmic factors in problem quantities (including dimensional factors and the time horizon $T$) and hence cannot match known asymptotics \cite{ljung1998system,Ljun1980} even up to constant factors. For the case of time-invariant regression functions, the authors in \cite{ziemann2024sharp} removed these logarithmic factors via a mixed-tail generic chaining argument.

%dont call it complex (im the author so I cant call my own stuff complicated lol). I'm gonna leave this here as a comment haha

%how about involved

%im leaving now i dont ruin your wriitng

%also lol

In this paper, we pursue a simpler approach and provide the first  
rate-optimal non-asymptotic prediction error bound for a relatively general class of \textit{time-varying} parametric predictor models. Our model class is rich enough to allow for ARMA models of the form \eqref{eq:arma} that satisfy a certain identifiability condition. Similar to \cite{Ziemann2022}, our approach is based on the martingale offset complexity introduced to the statistical literature by \cite{liang2015learning}. We arrive at our result by providing a refined analysis of this complexity notion for models of the form \eqref{model}. An informal version of our main result is presented next.

\begin{namedtheorem}[Informal Version of Theorem~\ref{theorem:main_result}]
Given data from a sufficiently stable system, for a wide range of identifiable models $f_t(\cdot,\theta_{\star})$, the mean-squared prediction error corresponding to any least-squares estimate $\hattheta\in\argmin_{\theta\in\sfM}L_T(\theta)$ satisfies:
\begin{align*}
    \textup{Mean-Squared Prediction Error}(\hat{\theta})\leq\;\frac{\textup{parameter dimension}\times\textup{noise level}}{\textup{number of samples}}+\textup{higher-order terms}.
\end{align*}
\end{namedtheorem}

The above statistical rate matches known asymptotics \cite{ljung1998system,Ljun1980} up to constant factors and higher-order terms that become negligible for a large enough sample size $T$. The requirement that $T$ is larger than a so-called burn-in time is necessary to establish several components of our result, such as  persistence of excitation. We note that the stability properties of the model, which are measured via the stochastic dependency of the input process $\{X_t\}_{t=0}^{T-1}$, affects only the burn-in time of our result. 

In the next section, we formally present our mathematical assumptions and the problem formulation. In Section~\ref{Non-asymptotic_Analysis_of_Parametric_Time_Series_Least-Squares_Regression}, we introduce our main result, a proof sketch of which is given in Section~\ref{Proof_of_Theorem_1}. In Section~\ref{sec:examples}, we apply our main result to scalar ARMA models. Full proofs of all components of the main theorem's proof can be found in the Appendix.
\smallskip\\
\noindent\textbf{Notation.} The norm $\norm{\cdot}$ is the Euclidean norm whenever it is applied to vectors and the spectral norm whenever it is applied to matrices. Moreover, $\setS^{d-1}$ denotes the unit sphere in $\setR^{d}$, and $\setB_r^{d}$ the Euclidean ball of radius $r$ in $\setR^d$. We use $\mathbb{I}_d$ to denote the identity matrix of size $d$ and $\tr(A)$ to denote the trace of any square matrix $A\in\setR^{d\times d}$. Expectation and probability with respect to all the randomness of the underlying probability space are denoted by $\E$ and $\P$, respectively. Expectation with respect to a random variable $X$ is denoted by $\E_{X}$. Conditional expectation of a random variable $X$ with respect to an event $\calE$ and a
$\sigma$-field $\calF$ is denoted by $\E[X|\calE]$ and $\E[X|\calF]$, respectively. For any event $\calE$, we define $\mathds{1}_{\calE}$ as the indicator function of $\calE$, which takes value $1$ when the event occurs and $0$ otherwise. If $g(\cdot)$, $h(\cdot)$ are functions defined on some unbounded subset of the positive real numbers and $h(x)$ is strictly positive for all large enough values of $x$, we write $g=\calO(h)$ if there exists $x_0\in\setR$ such that $\limsup_{x\to x_0}\abs{g(x)/h(x)}<\infty$. 

\section{Problem Formulation}\label{Problem_Formulation}
Consider the predictor model \eqref{model}, where the input variables $X_t$ take values in $\sfX\subset\setR^{\dx}$, whereas the output
and noise variables, denoted by $Y_t$ and $W_t$, respectively, take values in $\mathbb{R}$. For each $t$, the regression function $f_t:\setR^{\dx}\times\setR^{\dtheta}\to\setR$ is known and depends on the input $X_t$ and the unknown parameter $\theta_{\star}$. The parameter $\theta_{\star}$ is assumed to belong to some known and compact parameter class $\sfM\subseteq\setB_{B_{\theta}}^{\dtheta}$, where $B_{\theta}$ is a positive constant. % We emphasize that past input-output pairs $(X_{t},Y_{t})$ may influence future input-output pairs $(X_{t+\tau},Y_{t+\tau})$ (i.e., with $\tau\in \N$). 

Before formalizing our problem, we introduce a few further assumptions about model \eqref{model} and the parameter class $\sfM$. We start by characterizing the stochastic dependency of the input process $\{X_t\}_{t=0}^{T-1}$, which can be thought of as a measure of the stability of model \eqref{model}. Let us first state the main definition we will need for this characterization. 

\begin{definition}[Dependency Matrix]\label{definition:dependency_matrix}\cite[Section 2]{Samson2000}
Let $\{Z_t\}_{t=0}^{T-1}$ be a stochastic process with joint distribution $\sfP_{\sfZ}$. For each pair $(i,j)$, let $\sfP_{Z_{i:j}}$ denote the joint distribution of $\{Z_t\}_{t=i}^j$ and $\calZ_{ij}:=\sigma(Z_i,\ldots,Z_j)$ the $\sigma$-algebra generated by $\{Z_t\}_{t=i}^{j}$. The dependency matrix of $\{Z_t\}_{t=0}^{T-1}$ is the matrix $\Gammadep(\sfP_{\sfZ}):=\{\Gamma_{ij}\}_{i,j=0}^{T-1}\in\setR^{T\times T}$, where: 
\begin{equation*}
    \Gamma_{ij}=\sqrt{2\sup_{\substack{A\in\calZ_{0:i}\\B\in\calZ_{j:T-1}}}\abs{\sfP_{Z_{j:T-1}}(B|A)-\sfP_{Z_{j:T-1}}(B)}},
\end{equation*}
for $i<j$, $\Gamma_{ii}=1$, and $\Gamma_{ij}=0$, for $i>j$.    
\end{definition}

Let $\sfP_{\sfX}$ denote the joint distribution of the input process $\{X_t\}_{t=0}^{T-1}$. We can measure the dependency of $\{X_t\}_{t=0}^{T-1}$ via the norm $\norm{\Gammadep(\sfP_{\sfX})}$ of its dependency matrix. Notice that $\Gammadep(\sfP_{\sfX})$ always satisfies $1\leq\norm{\Gammadep(\sfP_{\sfX})}\leq cT$, for some $c>0$. The lower bound of $\norm{\Gammadep(\sfP_{\sfX})}$ corresponds to independent input processes, whereas the upper bound corresponds to fully dependent input processes (i.e., processes with $X_t=X_{t+1}$, for all $t=0,\ldots,T-2$). Our results apply to processes for which $\norm{\Gammadep(\sfP_{\sfX})}^2$ grows sublinearly in $T$, as formalized in the following assumption. 

\begin{assumption}\label{assumption:input_dependency}
There exist $b_1>0$ and $b_2\in[0,1)$ such that $\norm{\Gammadep(\sfP_{\sfX})}^2\leq b_1T^{b_2}$.
\end{assumption}

Assumption~\ref{assumption:input_dependency} holds for a large family of input processes $\{X_t\}_{t=0}^{T-1}$ including, e.g., geometrically $\phi$-mixing processes \cite{Samson2000}, processes that satisfy Doeblin's condition \cite{Meyn2012,Samson2000}, and stationary time-homogeneous Markov chains \cite{Ziemann2022} (see \cite{Ziemann2022} for details). In the context of stable linear dynamical systems with bounded noise, it has been shown that the spectral norm of the dependency matrix $\Gammadep(\sfP_{\sfX})$ is uniformly bounded (i.e., $b_2=0$) \cite{Ziemann2022}, which implies an intuitive connection between stability and dependency in the process $\{X_t\}_{t=0}^{T-1}$.

\begin{assumption}\label{assumption:noise}
For each $t$, let $\calF_t:=\sigma(X_0,\ldots,X_{t+1},$ $W_0,\ldots,W_t)$ be the $\sigma$-field generated by the inputs $X_0,\ldots,X_{t+1}$ and the noise variables $W_0,\ldots,W_t$. For every $t$, the noise variable $W_t$ is $\sigma_w^2$-conditionally sub-Gaussian with respect to  $\calF_{t-1}$, that is:
\begin{equation*}
    \E[e^{\lambda W_t}|\calF_{t-1}]\ \leq e^{\frac{\lambda^2\sigma_w^2}{2}},
\end{equation*}
for all $\lambda\in\setR$, for some $\sigma_w>0$.
\end{assumption}

Assumption~\ref{assumption:noise} is satisfied if the noise variables $W_t$ are i.i.d. zero-mean Gaussian with variance $\sigma_w^2$ and independent of the inputs $X_0,\ldots,X_{t}$. In addition, it is satisfied by a large number of non-Gaussian random variables $W_t$ \cite{Wainwright2019}; it is also standard in prior work \cite{Ziemann2022a,Ziemann2022,Ziemann2023}.

\begin{assumption}\label{assumption:regression_function}
For each $t$, the regression function $f_t(\cdot,\cdot)$ is twice differentiable with respect to its second argument. Moreover, there exist $L_1,L_2>0$ such that the partial gradients $\gradtheta f_t(\cdot,\cdot)$ and the partial Hessians $\gradtheta^2 f_t(\cdot,\cdot)$ satisfy $\norm{\gradtheta f_t(x,\theta)}\leq L_1$ and $\norm{\gradtheta^2 f_t(x,\theta)}\leq L_2$, respectively, for all $(x,\theta)\in\sfX\times\sfM$. In addition, the partial Hessians $\gradtheta^2 f_t(\cdot,\cdot)$ are $L_3$-Lipschitz continuous in their second argument with respect to the norm $\norm{\cdot}$.
\end{assumption}

Note that for all functions $f_t(\cdot,\cdot)$ that are three times differentiable with respect to their second argument, Assumption~\ref{assumption:regression_function} trivially holds if $\sfX$ is bounded given that $\sfM\subseteq\setB_{B_{\theta}}^{\dtheta}$ is bounded. One expects that our results extend to unbounded inputs via a truncation argument, see for instance \cite[Section 5.1]{Ziemann2022}. We leave a thorough analysis of this case for future work. 

\begin{assumption}[Positive Definite Information Matrix]\label{assumption:excitation}
There exists $\lambda_0>0$ such that:
\begin{equation*}\label{excitation_condition}
    \E\bracks{\frac{1}{T}\sumt \gradtheta f_t(X_t,\theta_{\star}) \gradtheta^{\T}  f_t(X_t,\theta_{\star})}\succeq\lambda_0\mathbb{I}_{d_{\theta}}.
\end{equation*}
%The expectation on the right-hand side of \eqref{excitation_condition} is with respect to the random variables $\barX_0,\ldots,\barX_{T-1}\sim\sfP_{\sfX}$.
\end{assumption}

Assumption~\ref{assumption:excitation} imposes a minimal noise excitation condition, quantifying the notion of persistence of excitation \cite{ljung1998system}. Put differently, it asks that the parameter $\theta_\star$ is identifiable (in the second-order sense). We note in passing that analogous conditions are employed in recent related work (see, e.g., \cite{Mania2022,Kowshik2021,Ziemann2022}). %\charis{Comment further.}

\begin{assumption}[Quadratic Identifiability]\label{assumption:expansivity}
There exists $a>0$ such that for every $\theta\in\sfM$:
\begin{equation}\label{expansivity_condition}
    \norm{\theta-\theta_{\star}}^2\leq a\E\bracks{\frac{1}{T}\sumt(f_t(X_t,\theta)-f_t(X_t,\theta_{\star}))^2}.
\end{equation}
%The expectation on the right-hand side of \eqref{expansivity_condition} is with respect to the random variables $\barX_0,\ldots,\barX_{T-1}\sim\sfP_{\sfX}$.
\end{assumption}

Assumption~\ref{assumption:expansivity} imposes a regularity condition on the regression functions $f_t(\cdot,\cdot)$ with respect to the parameter space. More specifically, it quantifies the growth of the prediction error as quadratic in the parameter error. We point out that condition \eqref{expansivity_condition} is weaker than the global positive-definiteness condition:
\begin{equation*}
    \E\bracks{\frac{1}{T}\sumt \gradtheta f_t(X_t,\theta) \gradtheta^{\T} f_t(X_t,\theta)}\succeq\delta\mathbb{I}_{d_{\theta}},
\end{equation*}
which is often assumed in the asymptotic literature \cite{Ljun1980}, for all $\theta\in\sfM$, for some $\delta>0$. Moreover, note that Assumption~\ref{assumption:expansivity} always holds for linear dynamical systems as well as generalized linear models satisfying a certain expansivity condition (see, e.g., \cite{Sattar2022,Foster2020,Kowshik2021}).

The goal of system identification can often be cast as to identify the parameter $\theta_{\star}$, specifying the data-generating distribution in \eqref{model},
from sequential input-output data $(X_0,Y_0),\ldots,(X_{T-1},Y_{T-1})$ \cite{Box2015}. In this paper, we analyze the finite-sample performance of the regression functions $f_t(\cdot,\hattheta)$, where $\widehat \theta$ satisfies:
\begin{equation}\label{theta_hat}
    \hat{\theta}\in\argmin_{\;\;\quad \theta\in\sfM}\frac{1}{T}\sum_{t=0}^{T-1}(f_t(X_t,\theta)-Y_t)^2.
\end{equation}
In particular, we are interested in providing an upper bound for the (mean-squared) prediction error of the models $f_t(\cdot,\hattheta)$, given by:
\begin{equation}\label{excess}
   \E\bracks{\frac{1}{T}\sumt(f_t(\barX_t,\hattheta)-f_t(\barX_t,\theta_{\star}))^2}.
\end{equation}
Herein, we use $\{\barX_t\}_{t=0}^{T-1}$ to denote a fresh sample drawn from $\sfP_{\sfX}$ independently of $\{X_t\}_{t=0}^{T-1}$. We formalize the problem in the following statement.

\begin{problem}[Rate-Optimal Non-asymptotic Analysis of the Quadratic Prediction Error Method]\label{problem}
Assume that $\theta_{\star}$ in predictor model \eqref{model} is unknown. Consider a finite number $T\in\setN_+$ of sequential input-output data $(X_0,Y_0),\ldots,(X_{T-1},Y_{T-1})$ generated by model \eqref{model} and let $\hat{\theta}$ satisfy \eqref{theta_hat}. Provide bounds $T_0$ and $\e(T)$ such that if $T\geq T_0$, then:
\begin{equation*}
    \E\bracks{\frac{1}{T}\sumt(f_t(\barX_t,\hattheta)-f_t(\barX_t,\theta_{\star}))^2}\leq\e(T).
\end{equation*}
The bounds $T_0$ and $\e(T)$ may also depend on the parameters $\dtheta$, $\sigma_w$, $B_{\theta}$, $L_1$, $L_2$, $L_3$, $\lambda_0$, $a$, $b_1$, and $b_2$. Moreover, the prediction error bound $\e(T)$ should match known asymptotics up to constant factors in its leading term (see Remark~\ref{remark:known_asymptotics} for details).
\end{problem}

\begin{remark}\label{remark:known_asymptotics}
We refer to non-asymptotic rates for the prediction error \eqref{excess} as optimal if they match known asymptotics up to constant factors and higher-order terms. In particular, existing results for the quadratic prediction error method from the asymptotic literature \cite{ljung1998system,Ljun1980} guarantee that $\sqrt{T} (\hattheta - \theta_\star) $ converges in distribution to $\mathcal{N}(0,\mathcal{I}^{-1}(\theta_\star))$, where:
\begin{equation*}
\mathcal{I}(\theta_\star) :=  \frac{1}{\sigma_w^2}\E\bracks{\frac{1}{T}\sumt \nabla_{\theta}f_t(X_t,\theta_{\star})  \nabla_{\theta}^{\T} f_t(X_t,\theta_{\star}) }  
\end{equation*}
is the Fisher information matrix. An informal calculation---ignoring the higher-order terms in Taylor's theorem---suggests that the prediction error can be written as follows:
\begin{align}\label{eq:informalcramerrao}
     &\E\bracks{\frac{1}{T}\sumt(f_t(\barX_t,\hattheta)-f_t(\barX_t,\theta_{\star}))^2}\nonumber\\
     &\approx \E\bracks{(\hattheta - \theta_\star)^{\T} (\sigma_w^2 \mathcal{I}(\theta_\star))  (\hattheta - \theta_\star)}\nonumber\\
     &= \E\tr \left(  \sigma_w^2 \mathcal{I}(\theta_\star)  (\hattheta - \theta_\star) (\hattheta - \theta_\star)^{\T} \right).
\end{align}
Under suitable regularity conditions, we can deduce that the expectation of the trace on the right-hand side of \eqref{eq:informalcramerrao} asymptotically converges to $\frac{\sigma^2_w d_\theta }{T}$, that is:
\begin{align*}
    T^{-1} \E \tr \left(  \sigma_w^2 \mathcal{I}(\theta_\star)  (\hattheta - \theta_\star)  (\hattheta - \theta_\star)^{\T}  \right)
    \to \tr \left(  \sigma_w^2 \mathcal{I}(\theta_\star)  \mathcal{I}^{-1}(\theta_\star)  \right) = \sigma^2_w d_\theta. 
\end{align*}
In light of the above result, our goal is to obtain a rate of convergence that decays as fast as $\frac{ c \sigma^2_w d_\theta }{T}$, for some universal constant $c>0$.
\end{remark}

\section{Optimal Non-asymptotic Rates for the Quadratic Prediction Error Method}\label{Non-asymptotic_Analysis_of_Parametric_Time_Series_Least-Squares_Regression}
In this section, we present our main result, which is a rate-optimal bound for the prediction error \eqref{excess} of the models $f_t(\cdot,\hattheta)$, where $\hattheta$ is an estimate of the true parameter $\theta_{\star}$, satisfying \eqref{theta_hat}. Before we state our main theorem, let us note that herein, $\poly_{\psi}$ denotes a polynomial of degree of order $\psi$ in its arguments.  
  
\begin{theorem}[Optimal Non-asymptotic Rates for the Quadratic Prediction Error Method]\label{theorem:main_result}
Consider the predictor model \eqref{model} and the parameter class $\sfM$ under Assumptions~\cref{assumption:input_dependency,assumption:noise,assumption:regression_function,assumption:excitation,assumption:expansivity}. Fix any $\gamma\in(0,1/2)$ and let $\hattheta$ satisfy \eqref{theta_hat}. Then, there exist:
\begin{subequations}
\begin{align}
    \label{burn_in_time_1}
    T_1&:=\poly_{\frac{1}{1-b_2}}(d_{\theta},L_1,a,b_1,1/(1-b_2)), \\
    \label{burn_in_time_2}
    T_2&:=\poly_{\frac{1}{1-b_2}}(d_{\theta},\sigma_w,\Btheta,L_1,1/\lambda_0,b_1,1/(1-b_2)), \\
    \label{burn_in_time_3}
    T_3&:=\poly_{\frac{1}{1-2\gamma}}(d_{\theta},\sigma_w,\Btheta,L_1,L_2,L_3,a,1/(1-2\gamma)),
\end{align}
\end{subequations}
and a universal constant $c>0$ such that if $T\geq\max\{T_1,T_2,T_3\}$, we have:
\begin{equation}\label{prediction_error_bound}
     \E\bracks{\frac{1}{T}\sumt(f_t(\barX_t,\hattheta)-f_t(\barX_t,\theta_{\star}))^2}\leq \frac{c d_\theta\sigma_w^2}{T}+\frac{B}{T^{1+\gamma}},
\end{equation}
where $B=2L_1^2\Btheta^2+16$.
\end{theorem}

The exact expressions of the polynomials $T_1$, $T_2$, and $T_3$ of Theorem~\ref{theorem:main_result} are given in the Appendix. 

\begin{remark}[Result interpretation]\label{remark:main_result_interpretation}
Observe in \eqref{prediction_error_bound} that for sufficiently large sample size $T$, the least-squares prediction error decays at a rate of $\calO(T^{-1})$. In particular, the leading term in \eqref{prediction_error_bound} is determined by the signal-to-noise ratio (SNR) of model \eqref{model}, which is defined as $\mathsf{SNR}=\sigma_w^2/T$. Notice that the longer the predictor model is excited by noise and the smaller the sub-Gaussian parameter $\sigma_w$ is, the smaller the prediction error bound becomes. We note that this rate is optimal in the sense that it matches known asymptotics up to constant factors in its leading term  (see Remark~\ref{remark:known_asymptotics}), after a finite burn-in time $T_0:=\max\{T_1,T_2,T_3\}$. The burn-in time grows polynomially in: i) the parameter dimension $d_{\theta}$, ii) the sub-Gaussian parameter $\sigma_w$, iii) the noise bound $\Btheta$, iv) the dependency parameter $b_1$, v) the smoothness parameters $L_1$, $L_2$, $L_3$, $a$, and vi) the inverse of the noise excitation constant $\lambda_0$. Notice also the exponential growth of $T_0$ in the dependency parameter $b_2$. The parameter $b_2$ is typically zero for exponentially stable dynamical systems (consider, e.g., exponentially stable ARMA models, cf. \cite{Ziemann2022} for the case of autoregressive models). Nonetheless, improving this growth rate is an interesting future research direction.
\end{remark}

In the following section, we sketch the proof steps of Theorem~\ref{theorem:main_result}.

\section{Proof Sketch of Theorem~\ref{theorem:main_result}}\label{Proof_of_Theorem_1}
In this section, we present the main proof steps of Theorem~\ref{theorem:main_result}, which provides us with optimal non-asymptotic rates for the quadratic prediction error method.

%The main idea is to derive an upper bound on the prediction error of the least-squares estimator in terms of the corresponding martingale offset, and then derive a rate-optimal bound for that offset, refining results from prior work.

A key quantity appearing in our analysis is the martingale offset corresponding to a parameter $\theta\in\sfM$, which can be thought of as a measure of the complexity of the corresponding regression functions $f_t(\cdot,\theta)$. To formally define the martingale offset, let us first introduce relevant notation. Let $\{W_t\}_{t=0}^{T-1}$ denote the noise sequence corresponding to the input-output data $(X_0,Y_0),\ldots,(X_{T-1},Y_{T-1})$, i.e., let $W_t=Y_t-f_t(X_t,\theta_{\star})$, for all $t=0,\ldots,T-1$. Moreover, consider the shifted process $\{g_t(X_t,\theta)\}_{t=0}^{T-1}$, where $g_t(X_t,\theta)=f_t(X_t,\theta)-f_t(X_t,\theta_{\star})$, for all $t=0,\ldots,T-1$. For any $\theta\in\sfM$, the \textit{martingale offset} corresponding to the parameter $\theta$ is defined as:
\begin{equation*}\label{martingale_offset}
    M_T(\theta) = \frac{1}{T}\sumt \left(4W_tg_t(X_t,\theta)-g_t^2(X_t,\theta)\right).
\end{equation*}
The above definition is motivated by the  \textit{martingale offset complexity} $\sup_{\theta\in\sfM}M_T(\theta)$, which is employed in previous works \cite{Liang2015,Ziemann2022a,Ziemann2022}. As we will see in the analysis that follows, deriving a bound on the expected martingale offset $\E M_T(\hattheta)$ of any least-squares estimate $\hattheta$, instead of the expected martingale offset complexity $\E[\sup_{\theta\in\sfM}M_T(\theta)]$, is essential for obtaining optimal finite-sample rates for the quadratic prediction error method.

In the theorem below, we present a bound for the prediction error of the models $f_0(\cdot,\hattheta),\ldots,$ $f_{T-1}(\cdot,\hattheta)$, conditioned on the given sample $\{(X_t,Y_t)\}_{t=0}^{T-1}$. Note that the following theorem is a modified version of \cite[Corollary 4.2]{Ziemann2022} for time-varying predictor models. We achieve the extension to the time-varying case by deriving concentration inequalities for the sum of time-varying functions of the input data, leveraging a result from \cite{Samson2000} (see Appendix~\ref{Proof_of_Theorem_2} for details).
\begin{theorem}\label{theorem:theorem2}
Consider the predictor model \eqref{model} and the parameter class $\sfM$ under Assumptions~\cref{assumption:input_dependency,assumption:regression_function,assumption:expansivity}. Fix any $\gamma\in[0,1)$ and let $\hattheta$ satisfy \eqref{theta_hat}. Then, there exists $T_1$, defined as in \eqref{burn_in_time_1}, such that if $T\geq T_1$, we have:
\begin{align}\label{inequality:theorem2}
    &\E_{\barX_{0:T-1}}\bracks{\frac{1}{T}\sumt(f_t(\barX_t,\hattheta)-f_t(\barX_t,\theta_{\star}))^2}\leq8 M_T(\hat{\theta})+\frac{2L_1^2B_{\theta}^2}{T^{1+\gamma}},
\end{align}
where $\barX_{0:T-1}=(\barX_0,\ldots,\barX_{T-1})$.
\end{theorem}
Given \eqref{theta_hat}, by taking the expectation over the sample $\{(X_t,Y_t)\}_{t=0}^{T-1}$, \eqref{inequality:theorem2} yields:
\begin{align}\label{theorem:theorem1-11}
    &\E\bracks{\frac{1}{T}\sumt(f_t(\barX_t,\hattheta)-f_t(\barX_t,\theta_{\star}))^2}\leq8 \E M_T(\hat{\theta})+\frac{2L_1^2B_{\theta}^2}{T^{1+\gamma}},
\end{align}
for any $\gamma\in[0,1)$. Moreover, Assumption~\ref{assumption:expansivity} can be invoked to obtain the parameter error bound:
\begin{equation}\label{inequality:theorem2_old}
    \norm{\hat{\theta}-\theta_{\star}}^2\leq 8a M_T(\hat{\theta})+\frac{2aL_1^2B_{\theta}^2}{T},
\end{equation}
since $T\leq T^{1+\gamma}$, for all $\gamma\in[0,1)$.

Employing the result from \cite[Lemma 10]{Ziemann2022a} to bound the expected martingale offset complexity $\E [\sup_{\theta\in\sfM}M_T(\theta)]$, 
inequality \eqref{theorem:theorem1-11} directly provides us with a prediction error bound of order $\calO(\log T/T)$, after a finite burn-in time $T_1$. We note that this is the first non-asymptotic result for predictor models of the form \eqref{model}, where the regression functions $f_t(\cdot,\cdot)$ are time-varying.

In the rest of this section, our goal is to improve upon that rate and ensure an optimal convergence rate of $\calO(T^{-1})$, after a longer but finite burn-in time. Recall that herein optimality of rate implies matching existing results from the asymptotic literature, modulo constant factors and higher-order terms (see Remark~\ref{remark:known_asymptotics}). Our refinement of the prediction error bound resulting from Theorem~\ref{theorem:theorem2} consists of three distinct steps:
\begin{enumerate}[label=\roman*)]
\item First, we derive an upper bound for $\E M_T(\hattheta)$ by using the Taylor expansion of the models $f_t(X_t,\hattheta)$ around $\theta_{\star}$. This bound depends on a ``linearized'' term and higher-order terms related to the parameter error $\norm{\hattheta-\theta_{\star}}$. 
\item Second, we provide a rate-optimal bound which scales like $d_\theta \sigma_w^2 T^{-1}$ for the ``linearized'' term, by leveraging ideas from linear system identification.
\item Third, we combine our bound for the ``linearized'' term with faster decaying bounds for  the higher-order terms to obtain a refined bound for $\E M_T(\hattheta)$. The main idea is to bound the higher-order terms employing the nearly optimal bounds resulting from Theorem~\ref{theorem:theorem2}. Owing to the higher order of these components, a careful analysis does not degrade the leading $d_\theta \sigma_w^2 T^{-1}$-order term of the linearized component.
\end{enumerate}

Putting everything together, we provide the first optimal non-asymptotic rates for the quadratic prediction error method. For clarity of presentation, we separately analyze the aforementioned proof steps.
\smallskip\\
\noindent\textbf{Step I: Bounding $\E M_T(\hattheta)$ via Taylor expansion.} 
By Taylor's theorem with remainder, for each $t$, we have:
\begin{align}\label{lemma:taylor_bound-2}
f_t(X_t,\hat{\theta}) &= f_t(X_t,\theta_{\star})+Z_t^{\T}(\hat{\theta}-\theta_{\star})+\frac{1}{2}(\hat{\theta}-\theta_{\star})^{\intercal}V_t(\hat{\theta}-\theta_{\star}),    
\end{align}
where $Z_t=\nabla_{\theta}f_t(X_t,\theta_{\star})$ and $V_t=\nabla_{\theta}^2f_t(X_t,\tilde{\theta}_t)$, with $\tilde{\theta}_t=\alpha_t\hat{\theta}+(1-\alpha_t)\theta_{\star}$, for some $\alpha_t\in[0,1]$. Exploiting the Taylor expansion of each $f_t(X_t,\hattheta)$ from \eqref{lemma:taylor_bound-2}, we can prove the following lemma.
\begin{lemma}\label{lemma:taylor_bound}
Consider the predictor model \eqref{model} and the parameter class $\sfM$ under Assumption~\cref{assumption:regression_function}. Moreover, let $\hattheta$ satisfy \eqref{theta_hat}, and for each $t$, consider the Taylor expansion of $f_t(X_t,\hattheta)$ around $\theta_{\star}$ given in \eqref{lemma:taylor_bound-2}. Then, the martingale offset of $\hattheta$ satisfies:
\begin{align}\label{inequality:lemma_taylor_bound}
    M_T(\hattheta)\leq &\;\thickbar{M}_T(\hattheta)+\normbig{\frac{2}{T}\sumt W_tV_t}\norm{\hat{\theta}-\theta_{\star}}^2+\frac{L_2^2}{4}\norm{\hat{\theta}-\theta_{\star}}^4,
    \end{align}
where:
\begin{equation}\label{linearized_term}
    \thickbar{M}_T(\hattheta)= \frac{1}{T}\sum_{t=0}^{T-1}\Big[4W_tZ_t^{\T}(\hat{\theta}-\theta_{\star})-\frac{1}{2}(Z_t^{\T}(\hat{\theta}-\theta_{\star}))^2\Big].
\end{equation}
\end{lemma}

By taking the expectation over the sample $\{X_t,Y_t\}_{t=0}^{T-1}$ in \eqref{inequality:lemma_taylor_bound}, we obtain the following bound for the expected martingale offset of $\hattheta$:
\begin{align}\label{theorem:theorem2-2}
    \E M_T(\hattheta)\leq &\;\E\thickbar{M}_T(\hattheta)+\E\bracks{\normbig{\frac{2}{T}\sumt W_tV_t}\norm{\hat{\theta}-\theta_{\star}}^2}+\frac{L_2^2}{4}\E\norm{\hat{\theta}-\theta_{\star}}^4.
\end{align}

Notice that the bound for $\E M_T(\hattheta)$ in \eqref{theorem:theorem2-2} consists of the ``linearized'' term $\E\thickbar{M}_T(\theta)$ (note that the quadratic term on the right-hand side of \eqref{linearized_term} is negative)  and two higher-order terms depending on the parameter error $\norm{\hattheta-\theta_{\star}}$. In the next step of our proof, we focus on bounding the linearized component. 
\smallskip\\
\noindent\textbf{Step II: Bounding the ``linearized'' term $\E\thickbar{M}_T(\hattheta)$.}
In the following theorem, we provide a bound for the ``linearized'' term $\E\thickbar{M}_T(\theta)$ appearing on the right-hand side of \eqref{theorem:theorem2-2}. Our analysis employs tools for self-normalized martingales, similar to previous works in linear system identification (see, e.g., \cite{tsiamis2023statistical}). 

\begin{theorem}\label{theorem:expected_self_normalized_martingale}
Consider the predictor model \eqref{model} and the parameter class $\sfM$ under Assumptions~\cref{assumption:input_dependency,assumption:noise,assumption:regression_function,assumption:excitation}. Fix any $\gamma\in(0,1)$ and let $\hattheta$ satisfy \eqref{theta_hat}. Then, there exists $T_2$, defined as in \eqref{burn_in_time_2}, and a universal constant $c>0$ such that if $T\geq T_2$, we have:
%~\cref{assumption:noise,assumption:regression_function,assumption:excitation,assumption:input_dependency}
\begin{equation}\label{expected_self_normalized_martingale}
    \E \thickbar{M}_T(\hattheta)\leq \frac{cd_{\theta}\sigma_w^2}{T}+\frac{1}{T^{1+\gamma}}.
\end{equation}
\end{theorem}

Notice that the bound in \eqref{expected_self_normalized_martingale} decays at the optimal rate of $\frac{\sigma_w^2 d_\theta}{T}$ (cf. Remark~\ref{remark:known_asymptotics}), up to a constant factor $c>0$ and a higher-order term $1/T^{\gamma+1}$, which becomes negligible in finite time (set for instance $\gamma = 1/4)$. Next, we present the final step of our proof, which combines the bound \eqref{expected_self_normalized_martingale} with faster decaying bounds for the higher-order terms on the right-hand side of \eqref{theorem:theorem2-2}.
\smallskip\\
\noindent\textbf{Step III: Bounding $\E M_T(\hattheta)$ using the bound \eqref{expected_self_normalized_martingale} for $\E\thickbar{M}_T(\hattheta)$ and the bound \eqref{inequality:theorem2_old} for the higher-order terms in \eqref{theorem:theorem2-2}.} In Step II we provided a bound of order $\calO(T^{-1})$ for the ``linearized'' term $\E\thickbar{M}_T(\theta)$ appearing on the right-hand side of \eqref{theorem:theorem2-2}. To bound $\E M_T(\hattheta)$ at a rate of $\calO(T^{-1})$, it suffices to derive faster decaying bounds of order $\calO(T^{1/(1+\gamma)})$ for the higher-order terms, where $\gamma\in(0,1/2)$. Combining \eqref{expected_self_normalized_martingale} from Theorem~\ref{theorem:expected_self_normalized_martingale} and the parameter error bound given in \eqref{inequality:theorem2_old}, we obtain the following corollary.

\begin{corollary}\label{corollary:expected_M__hattheta_bound}
Consider the predictor model \eqref{model} and the parameter class $\sfM$ under Assumptions~\cref{assumption:input_dependency,assumption:noise,assumption:regression_function,assumption:excitation,assumption:expansivity}. Fix any $\gamma\in(0,1/2)$ and let $\hattheta$ satisfy \eqref{theta_hat}. Then, there exist  $T_1$, $T_2$, $T_3$, defined as in Theorem~\ref{theorem:main_result}, and a universal constant $c>0$ such that if $T\geq \max\{T_1,T_2,T_3\}$, we have:
\begin{equation}\label{expected_M__hattheta_bound}
    \E M_T(\hattheta)\leq \frac{cd_{\theta}\sigma_w^2}{T}+\frac{2}{T^{1+\gamma}}.
\end{equation}
\end{corollary}

Notice that the dominant term on the right-hand side of \eqref{expected_M__hattheta_bound} decays at a rate of $\frac{cd_\theta \sigma_W^2}{T}$, which is optimal up to a constant factor $c>0$ (see Remark~\ref{remark:known_asymptotics}). We note that the above bound improves upon the rate $\calO(\log T/T)$ that can been shown for the martingale offset complexity $\E [\sup_{\theta\in\sfM}M_T(\theta)]$ via maximal inequalities \cite{Ziemann2022a}. The key point here is that the offset process locally, once $\widehat \theta$ is sufficiently near $\theta_\star$, behaves like a linear offset process.
 
Combining \eqref{theorem:theorem1-11} with \eqref{expected_M__hattheta_bound} from Corollary~\ref{corollary:expected_M__hattheta_bound}, we complete the proof of \eqref{prediction_error_bound} in Theorem~\ref{theorem:main_result}. Next, we instantiate Theorem~\ref{theorem:main_result} to provide finite-sample guarantees for the quadratic prediction error method for AutoRegressive Moving Average (ARMA) models.

\section{Case Study: The ARMA Model}
\label{sec:examples}
In this section, we demonstrate the applicability of our rate-optimal non-asymptotic analysis of the quadratic prediction error method to scalar ARMA models. Our result relies on a standard analysis from \cite{Tsybakov2009,Davis2013} that allows converting any ARMA model into a predictor model of the form \eqref{model}. For completeness of presentation, we briefly review the conversion methodology and then present a rate-optimal non-asymptotic bound for a particular class of ARMA models. 

Consider the scalar ARMA$(p,q)$ model given by: 
\begin{equation}\label{eq:arma2}
   Y_t = \sum_{i=1}^p a_i^{\star} Y_{t-i}+ \sum_{j=0}^q b_j^{\star} W_{t-j},
\end{equation}
where the noise variables $W_t\in\setR$ are assumed to be independent and zero-mean, and the initial conditions are assumed to be zero, i.e., $Y_t=0$, $W_t=0$, for all $t<0$. Suppose that $b_0^{\star}=1$ and the parameter $\theta_{\star}:=\begin{bmatrix}a_1^{\star},\ldots,a_p^{\star},b_0^{\star},\ldots,b_q^{\star}\end{bmatrix}^{\T}\in\setR^{p+q+1}$ belongs to some known set $\sfM\subseteq\setB_{\Btheta}^{d_{\theta}}$, where $\Btheta$ is a positive constant. The assumption that $b_0^{\star}=1$ can always be ensured by providing additional artificial noise components of zero mean and 
variance, and applying linear transformations to the noise variables $W_t$ \cite{Davis2013}.  Let $z^{-1}$ denote the backward-shift operator, defined by $z^{-1}e_t:=e_{t-1}$, for any stochastic process $\{e_t\}_{-\infty}^{\infty}$. Powers of $z^{-1}$ are defined recursively by $z^{-(i+1)}e_t:=z^{-1}(z^{-i}e_t)$ so that $z^{-i}e_t=e_{t-i}$. It is straightforward to show that \eqref{eq:arma2} is equivalent to $A_{\theta_{\star}}(z^{-1})Y_t = B_{\theta_{\star}}(z^{-1})W_t$, where $A_{\theta_{\star}}(\cdot)$ and $B_{\theta_{\star}}(\cdot)$ are polynomials given by: 
\begin{equation*}\label{lagpolys_A}
\begin{aligned}
    A_{\theta_{\star}}(\lambda) = 1- \sum_{i=1}^p a_i^{\star} \lambda^i,\; B_{\theta_{\star}}(\lambda) = \sum_{j=0}^q b_j^{\star} \lambda^j,
\end{aligned}
\end{equation*}
respectively, for all $\lambda\in\setR$. For each $t$, let $\thickbar{\calF}_t:=\sigma(Y_0,\ldots,Y_t)$ be the $\sigma$-field generated by the outputs $Y_0,\ldots,Y_{t}$.  It is known \cite[Section 2.6]{Tsybakov2009} that the conditional expectation $\widehat Y_t:=\E [ Y_t | \thickbar{\calF}_{t-1}]$ satisfies:
\begin{equation}\label{yhat_definition}
    B_{\theta_{\star}}(z^{-1}) \widehat Y_t = [B_{\theta_{\star}}(z^{-1})-A_{\theta_{\star}}(z^{-1})]Y_t,
\end{equation}
for all $t=0,\ldots,T-1$. Hence, we can rewrite model \eqref{eq:arma2} in the predictor model form \eqref{model} with regression functions: 
\begin{equation}\label{arma_regression_functions}
    f_t(X_t,\theta_\star) := \widehat Y_t,
\end{equation}
where $X_t=\begin{bmatrix}Y_0,\ldots,Y_{t-1}\end{bmatrix}^{\T}$. The conditional expectations $\hat{Y}_0,\ldots,\hat{Y}_{T-1}$ can be computed recursively from \eqref{yhat_definition} with zero initial condition, i.e., $\hat{Y}_t=0$, for all $t<0$. We can similarly define the regression functions $f_t(\cdot,\theta)$ corresponding to any parameter $\theta$ in the class $\sfM$. 
In the corollary that follows, we combine Theorem~\ref{theorem:main_result} with the predictor model form of the ARMA model \eqref{eq:arma2} derived above and provide the first rate-optimal non-asymptotic prediction error bounds for ARMA models.

\begin{corollary}\label{corollary:ARMA}
Consider the predictor model form of the ARMA$(p,q)$ model \eqref{eq:arma2}, defined by \eqref{model} and \eqref{arma_regression_functions}, as well as the parameter class $\sfM$, under Assumptions~\cref{assumption:input_dependency,assumption:noise,assumption:regression_function,assumption:excitation,assumption:expansivity}. Fix any $\gamma\in(0,1/2)$ and let $\hattheta$ satisfy \eqref{theta_hat}. Then, there exist $T_1$, $T_2$, $T_3$, defined as in \Cref{theorem:main_result}, and a universal constant $c>0$ such that if $T\geq\max\{T_1,T_2,T_3\}$, we have:
\begin{align*}
    &\E\left[\frac{1}{T}\sumt(f_t(X_t,\hattheta)-f_t(X_t,\theta_{\star}))^2\right]\leq \frac{c(p+q)  \sigma_w^2}{T}+\frac{B}{T^{\gamma+1}},
\end{align*}
where $B=2L_1^2\Btheta^2+16$.
\end{corollary}

The proof of Corollary~\ref{corollary:ARMA}  follows directly from Theorem~\ref{theorem:main_result}, given the predictor model form of model \eqref{eq:arma2}.

Note that the above corollary applies to a particular class of ARMA models that satisfy Assumptions~\cref{assumption:input_dependency,assumption:noise,assumption:regression_function,assumption:excitation,assumption:expansivity}. Assumptions~\cref{assumption:input_dependency,assumption:noise,assumption:regression_function,assumption:excitation} are relatively benign for this example, and hold as long as the noise sequence $\{W_t\}_{t=0}^{T-1}$ is bounded and the system \eqref{eq:arma2} is stable. For Assumption 1, see e.g. \cite{Ziemann2022} for the case of $B_{\theta_{\star
}}(\lambda)=1$. Assumption 2 is true by construction of the regression functions \eqref{arma_regression_functions} corresponding to model \eqref{eq:arma2} as well as the hypothesis of bounded noise. Assumption 3 can be verified via arguments entirely analogous to those in \cite{caines1976prediction} as long as $\{W_t\}_{t=0}^{T-1}$ is bounded and the system \eqref{eq:arma2} is stable. Sufficient conditions for guaranteeing Assumption \cref{assumption:excitation}, related to the roots of the polynomials $A_{\theta_{\star}}(\lambda)$ and $B_{\theta_{\star}}(\lambda)$, can be found in \cite{klein1996fisher}. Assumption~\ref{assumption:expansivity} restricts our result to a specific class of quadratically identifiable ARMA models (see \eqref{expansivity_condition}). As previously explained in Section~\ref{Problem_Formulation}, this assumption is weaker than the corresponding assumption made in the asymptotic literature for the quadratic prediction error method \cite{Ljun1980}. Exploring potential relaxations of the identifiability condition \eqref{expansivity_condition} is an interesting problem for future work.

\vspace{-8pt}

\section*{Acknowledgements}
Charis Stamouli and George J. Pappas acknowledge support from NSF award SLES-2331880. Ingvar Ziemann is supported by a Swedish Research Council international postdoc grant.

%\input{sections/unused/offset}
%\addcontentsline{toc}{section}{References}

%\bibliographystyle{plainnat}
%\bibliographystyle{icml2024}
\bibliographystyle{IEEEtran} % use IEEEtran.bst style
\bibliography{arxiv_version}

\appendix
\bigskip
\medskip
\noindent\textbf{\Large Appendix}
\section{Basic Definitions and Results}
In this subsection, we present a few basic lemmas that we will use in the proofs of our theorems in the subsequent subsections.

We first state a standard result from algebra. We include the proof for completeness.
\begin{lemma}\label{lemma:parallelogram}
For all $\alpha,\beta\in\setR$, we have:
\begin{align}\label{parallelogram}
    \frac{\beta^2}{2}-(\alpha-\beta)^2\leq\alpha^2\leq2\beta^2+2(\alpha-\beta)^2
\end{align}
\end{lemma}

\begin{proof}
Let us prove the right inequality in \eqref{parallelogram}. For all $x,y\in\setR$, we have $(x-y)^2\geq0$, which implies that:
\begin{align*}
     2xy\leq x^2+y^2.
\end{align*}
Setting $x=\beta$ and $y=\alpha-\beta$, we get:
\begin{equation}\label{lemma:parallelogram-1}
    2\beta(\alpha-\beta)\leq \beta^2+(\alpha-\beta)^2.
\end{equation}
For all $\alpha,\beta\in\setR$, we can write:
\begin{align*}
    \alpha^2=(\beta+(\alpha-\beta))^2=\beta^2+(\alpha-\beta)^2+2\beta(\alpha-\beta)\leq2\beta^2+2(\alpha-\beta)^2,
\end{align*}
where the inequality follows from \eqref{lemma:parallelogram-1}. Similarly, we can show the left inequality in \eqref{parallelogram}.
\end{proof}

Next, we state the definition of $\e$-nets and state a standard upper bound for the smallest possible cardinality of such sets corresponding to the ball $\setS^{d-1}$.

\begin{definition}[$\e$-net and Covering numbers]\label{definition:epsilon_net}\cite[Definitions 4.2.1, 4.2.2]{Vershynin2018}
Let $(\setX,d)$ be a compact metric space and fix $\e>0$. A subset $\calN_{\e}$ of $\setX$ is called an $\e$-net of $\setX$ if every point of $\setX$ is within radius $\e$ of a point of $\calN_{\e}$, that is:
\begin{equation}\label{definition_e_net}
    \sup_{x\in\setX}\inf_{x'\in\calN_{\e}}d(x,x')\leq\e.
\end{equation}
Moreover, the smallest possible cardinality of $\calN_{\e}$ for which \eqref{definition_e_net} holds is called the covering number at resolution $\e$ of $(\setX,d)$ and is denoted by $\calN(\e,\setX,d)$.
\end{definition}

We note that herein when we refer to $\e$-nets, we imply $\e$-nets of the smallest possible cardinality.

\begin{lemma}\label{lemma:covering_number}\cite[Corollary 4.2.13]{Vershynin2018}
For any $\e>0$, the covering numbers of $\setS^{d-1}$ satisfy:
\begin{equation}\label{covering_number}
    \calN(\e,\setS^{d-1},\norm{\cdot})\leq\pars{\frac{2}{\e}+1}^{d}.
\end{equation}
\end{lemma}

We can use Lemma~\ref{lemma:covering_number} to derive basic bounds for the covering numbers of the parameter class $\sfM$ and the function class:
\begin{align}\label{class_G}
    \calG:=\left\{g_{\theta}:(x_0,\ldots,x_{T-1})\mapsto \frac{1}{\sqrt{T}}(g_0(x_0,\theta),\ldots,g_{T-1}(x_{T-1},\theta))\,\Big|\, x_t\in\sfX, \forall t, \theta\in\sfM\right\},
\end{align}
where $g_t(x,\theta)=f_t(x,\theta)-f_t(x,\theta_{\star})$, for all $(x,\theta)\in\sfX\times\sfM$ and all $t=0,\ldots,T-1$. We present these bounds in the next two lemmas.

\begin{lemma}\label{lemma:M_covering_number}
For any $\e\in(0,\Btheta]$, the covering numbers of the parameter class $\sfM\subseteq\setB_{\Btheta}^{\dtheta}$ satisfy:
\begin{equation}\label{M_covering_number}
\calN(\e,\sfM,\norm{\cdot})\leq\pars{\frac{3\Btheta}{\e}}^{d_{\theta}}.
\end{equation}   
\end{lemma}

\begin{proof}
Fix any $\e'\in(0,1]$ and let $\calN_{\e'}$ be an $\e'$-net of $\setS^{d_{\theta}-1}$ with respect to the norm $\norm{\cdot}$. From Lemma~\ref{lemma:covering_number} we deduce that:
\begin{equation}\label{lemma:M_covering_number-1}
|\calN_{\e'}|\leq\pars{\frac{2}{\e'}+1}^{d_{\theta}}\leq\pars{\frac{3}{\e'}}^{d_{\theta}},   
\end{equation}
where the last inequality follows from the fact that $\e'\in(0,1]$. Given that $\sfM\subseteq\setB_{\Btheta}^{d_{\theta}}$, we have:
\begin{equation*}
\calN(\e,\sfM,\norm{\cdot})\leq \calN(\e,\setB_{\Btheta}^{\dtheta},\norm{\cdot}),
\end{equation*}
which implies that it suffices to obtain bounds for the covering numbers of $\setB_{\Btheta}^{d_{\theta}}$. Note that for every $\theta\in\setB_{\Btheta}^{d_{\theta}}\setminus\{0\}$, there exists $u:=\frac{1}{\norm{\theta}}\theta\in\setS^{d_{\theta}-1}$ (in the trivial case where $\theta=0$ we can set $u=0$). Then, by definition of $\calN_{\e'}$, there exists $u_i\in\calN_{\e'}$ such that $\norm{u-u_i}\leq\e'$. Hence, setting $\theta_i=\norm{\theta} u_i\in\setB_{\Btheta}^{d_{\theta}}$, we obtain:
\begin{align}\label{lemma:M_covering_number-2}
    \norm{\theta-\theta_i}=\norm{\theta}\norm{u-u_i}\leq\Btheta\e'.
\end{align}
For any $\e\in(0,\Btheta]$, we can set $\e'=\e/\Btheta$ in \eqref{lemma:M_covering_number-1} and \eqref{lemma:M_covering_number-2}, thus completing the proof of \eqref{M_covering_number}.
\end{proof}

\begin{lemma}\label{lemma:G_covering_number}
Consider the predictor model \eqref{model} and the parameter class $\sfM$ under Assumption~\ref{assumption:regression_function}. Moreover, let $\calG$ be as in \eqref{class_G}. Then, for any $\e\in(0,L_1\Btheta]$, the covering numbers of $\calG$ satisfy:
\begin{equation}\label{G_covering_number}
\calN(\e,\calG,\norm{\cdot}_{\infty})\leq\pars{\frac{3L_1\Btheta}{\e}}^{d_{\theta}}.
\end{equation}
\end{lemma}

\begin{proof}
Fix any $\e'\in(0,\Btheta]$ and let $\calN_{\e'}$ be an $\e'$-net of $\sfM$ with respect to the norm $\norm{\cdot}$. From Lemma~\ref{lemma:M_covering_number} we deduce that:
\begin{equation}\label{lemma:G_covering_number-1}
|\calN_{\e'}|\leq\pars{\frac{3\Btheta}{\e'}}^{d_{\theta}}.   
\end{equation}
By definition of $\calN_{\e'}$, for every $\theta\in\sfM$ there exists $\theta_i\in\calN_{\e'}$ such that $\norm{\theta-\theta_i}\leq\e'$. Let $g_{\theta}$ and $g_{\theta_i}$ denote the functions in $\calG$ corresponding to $\theta$ and $\theta_i$, respectively. Employing Assumption~\ref{assumption:regression_function}, we can write:
\begin{align}\label{lemma:G_covering_number-3}
\norm{g_{\theta}-g_{\theta_i}}_{\infty}&=\sup_{x_0,\ldots,x_{T-1}\in\sfX}\sqrt{\sumt\pars{\frac{1}{\sqrt{T}}(g_t(x_t,\theta)-g_t(x_t,\theta_i))}^2}\nonumber\\
&=\sup_{x_0,\ldots,x_{T-1}\in\sfX}\sqrt{\frac{1}{T}\sumt (f_t(x_t,\theta)-f_t(x_t,\theta_i))^2}\nonumber\\
&\leq \sqrt{\frac{1}{T}\sumt L_1^2\norm{\theta-\theta_i}^2}=L_1\norm{\theta-\theta_i}\leq L_1\e'.
\end{align}
We conclude that for every $g_{\theta}\in\calG$, there exists $g_{\theta_i}\in\calG$ such that \eqref{lemma:G_covering_number-3} holds. Hence, setting $\e'=\e/L_1$, from \eqref{lemma:G_covering_number-1} we  deduce that \eqref{G_covering_number} holds for any $\e\in(0,L_1\Btheta]$, thus completing the proof.
\end{proof}

\section{Proof of Theorem~\ref{theorem:theorem2}}\label{Proof_of_Theorem_2}

We start by presenting the following lemma, which combines a modified version of \cite[Theorem 5.1]{Ziemann2022} and \cite[Theorem 2]{Samson2000} for \textit{time-varying} parametric functions $g_t(\cdot,\theta)$. Although herein we focus on parametric functions, note that the following result directly extends to nonparametric functions $g_t(\cdot)$.

\begin{lemma}\label{lemma:analogue_of_Th.5.1}
For each $t$, let $g_t:\sfX\times\sfM\to\setR$ be such that $0\leq g_t(x,\theta)\leq C$, for all $(x,\theta)\in\sfX\times\sfM$, for some $C>0$. Then, for any $\lambda>0$ and $\theta\in\sfM$ we have:
\begin{align}\label{inequality:analogue_of_Th.5.1}
    &\E\exp\pars{-\lambda\sumt g_t(X_t,\theta)}\leq\exp\pars{-\lambda\sumt \E g_t(X_t,\theta)+\frac{\lambda^2\normGamma^2}{2}\sumt\limits\E g_t^2(X_t,\theta)}.
\end{align}
Moreover, for any $s>0$ and $\theta\in\sfM$ we have:
\begin{align}\label{inequality:analogue-of-Samson-Th.2-1}
&\P\pars{\sumt g_t(X_t,\theta)\geq\sumt\E g_t(X_t,\theta)+s}\leq\exp\pars{-\frac{s^2}{2C\normGamma^2(\sumt\E g_t(X_t,\theta)+s)}}
\end{align}
and:
\begin{align}\label{inequality:analogue-of-Samson-Th.2-2}
&\P\pars{\sumt g_t(X_t,\theta)\leq\sumt\E g_t(X_t,\theta)-s}\leq\exp\pars{-\frac{s^2}{2C\normGamma^2\sumt\E g_t(X_t,\theta)}}.
\end{align}

\end{lemma}

\begin{proof}
Fix any $\theta\in\sfM$ and define the finite function class:
\begin{align*}
    \calF=\{g_0(\cdot,\theta),\ldots,g_{T-1}(\cdot,\theta)\}.
\end{align*}
Moreover, let:
\begin{equation*}
    f_T(x_0,\ldots,x_{T-1})=\sumt g_t(x_t,\theta),
\end{equation*}
for all $x_0,\ldots,x_{T-1}\in\sfX$, and:
\begin{equation*}
    \alpha_{kt}=\left\{
            \begin{array}{lcr}
               1, \text{ if } k=t  \\
               0, \text{ else}
            \end{array}
            \right.,
\end{equation*}
for all $k,t\in\{0,\ldots,T-1\}$. Notice that:
\begin{equation*}
    f_T(x_0,\ldots,x_{T-1})=\sum_{k=0}^{T-1}\sumt \alpha_{kt}g_k(x_t,\theta),
\end{equation*}
for all $x_0,\ldots,x_{T-1}\in\sfX$, and:
\begin{equation*}
    \sum_{k=0}^{T-1}\alpha_{kt}=1,
\end{equation*}
for all $t=0,\ldots,T-1$. 
Fix any $x_0,\ldots,x_{T-1},x_0',\ldots,x_{T-1}'\in\sfX$.  Set $x_{0:T-1}=(x_0,\ldots,x_{T-1})$ and $x_{0:T-1}'=(x_0',\ldots,x_{T-1}')$. Furthermore, for each $t$, let $\alpha_t=(\alpha_{0t},\ldots,\alpha_{(T-1)t})$, $g(x_t,\theta)=(g_0(x_t,\theta),\ldots,$ $g_{T-1}(x_t,\theta))$, and $g(x_t',\theta)=$ $(g_0(x_t',\theta),\ldots,g_{T-1}(x_t',\theta))$. Then, we have:
\begin{align*}
    f_T(x_{0:T-1}')-f_T(x_{0:T-1})&=\sum_{k=0}^{T-1}\sumt \alpha_{kt}(g_k(x_t',\theta)-g_k(x_t,\theta))\\
    &=\sumt \alpha_t^{\T}(g(x_t',\theta)-g(x_t,\theta))\\
    &=\sumt \alpha_t^{\T}(g(x_t',\theta)-g(x_t,\theta))\mathds{1}_{x_t\neq x_t'}\\
    &\leq\sumt \alpha_t^{\T}g(x_t',\theta)\mathds{1}_{x_t\neq x_t'},
\end{align*}
where the last step follows from the fact that the functions $g_t(\cdot,\theta)$ and the coefficients $\alpha_{kt}$ are non-negative. Setting $\tilde{f}_T=-f_T$, we can similarly show that:
\begin{align*}
    \tilde{f}_T(x_{0:T-1}')-\tilde{f}_T(x_{0:T-1})\leq \sumt \alpha_t^{\T}g(x_t,\theta)\mathds{1}_{x_t\neq x_t'}.
\end{align*}
From this point onwards, the proof of \eqref{inequality:analogue_of_Th.5.1}, \eqref{inequality:analogue-of-Samson-Th.2-1}, and \eqref{inequality:analogue-of-Samson-Th.2-2} is the same as that of (3.12), (3.1), and (3.2), respectively, from \cite[Theorem 2]{Samson2000} for $Z=f_T(X_0,\ldots, X_{T-1})$, and thus is omitted.
\end{proof}

Employing \eqref{inequality:analogue_of_Th.5.1} from Lemma~\ref{lemma:analogue_of_Th.5.1}, we derive a lower isometry result that extends \cite[Theorem 5.2]{Ziemann2022} to \textit{time-varying} parametric regression functions $f_t(\cdot,\theta)$.

\begin{lemma}\label{lemma:analogue-of-Th.5.2}
Fix any $r\in(0,\sqrt{8}L_1\Btheta]$. Consider the predictor model \eqref{model} and the parameter class $\sfM$ under Assumptions~\cref{assumption:input_dependency,assumption:regression_function,assumption:expansivity}
%~\cref{assumption:input_dependency,assumption:regression_function,assumption:expansivity}
. For each $t$, define $g_t(x,\theta)=f_t(x,\theta)-f_t(x,\theta_{\star})$, for all $(x,\theta)\in\sfX\times\sfM$. Moreover, let $\calB_r=\left\{\theta\in\sfM\,\big|\,\frac{1}{T}\sumt\E g_t^2(X_t,\theta)\leq r^2\right\}$. Then, we have:
\begin{align}\label{inequality:analogue-of-Th.5.2}
&\Prob\pars{\inf_{\theta\in\sfM\setminus \calB_r}\pars{\frac{1}{T}\sumt g_t^2(X_t,\theta)-\frac{1}{8T}\sumt \E g_t^2(X_t,\theta)}\leq0}\leq\pars{\frac{8\sqrt{8}L_1\Btheta}{r}}^{2d_{\theta}}\exp\pars{-\frac{T^{1-b_2}}{8b_1L_1^4a^2}}.
\end{align}
\end{lemma}
\begin{proof}
Fix any $\theta\in\sfM$. Employing Assumptions~\ref{assumption:regression_function} and \ref{assumption:expansivity}, we can write:
\begin{align}\label{lemma:analogue-of-Th.5.2-1}
    &\frac{1}{T}\sumt\E g_t^4(X_t,\theta)=\frac{1}{T}\sumt \E(f_t(X_t,\theta)-f_t(X_t,\theta_{\star}))^4\leq L_1^4\norm{\theta-\theta_{\star}}^4\leq L_1^4a^2\pars{\frac{1}{T}\sumt \E g_t^2(X_t,\theta)}^2.
\end{align}
Employing \eqref{inequality:analogue_of_Th.5.1} from Lemma~\ref{lemma:analogue_of_Th.5.1} and \eqref{lemma:analogue-of-Th.5.2-1}, we can show that:
\begin{align}\label{lemma:analogue-of-Th.5.2-2}
    &\P\pars{\frac{1}{T}\sumt g_t^2(X_t,\theta)\leq\frac{1}{2T}\sumt\E g_t^2(X_t,\theta)}\leq \exp\pars{-\frac{T}{8L_1^4a^2\normGamma^2}},
\end{align}
similarly to the proof of \cite[Proposition 5.1]{Ziemann2022}. Let $\calG$ be as in \eqref{class_G}
and consider its star-hull:
\begin{align*}
    \starG:=\bigg\{&\thickbar{g}_{\theta}:(x_0,\ldots,x_{T-1})\mapsto \frac{1}{\sqrt{T}}(\thickbar{g}_0(x_0,\theta),\ldots,\thickbar{g}_{T-1}(x_{T-1},\theta))\,\Big|\,\nonumber\\
    & \thickbar{g}_t(x_t,\theta)=\alpha g_t(x_t,\theta),x_t\in\sfX, \forall t, \alpha\in[0,1], \theta\in\sfM\bigg\}.
\end{align*}
Let $\calG_{\star,r}=\Big\{\thickbar{g}_{\theta}\in\starG\,\big|\,\frac{1}{T}\sumt\E \thickbar{g}_t^2(X_t,\theta)\leq r^2\Big\}$ and let $\calN_{r}$ denote a $r/\sqrt{8}$-net of $\partial\calG_{\star,r}:=\left\{\thickbar{g}_{\theta}\in\starG\,\big|\,\frac{1}{T}\sumt\E \thickbar{g}_t^2(X_t,\theta)= r^2\right\}$ with respect to the norm $\norm{\cdot}_{\infty}$. For every $\thickbar{g}_{\theta}\in\starG$, there exist $\alpha\in[0,1]$ and $g_{\theta}\in\calG$ such that $\thickbar{g}_{\theta}=\alpha g_{\theta}$. Hence, employing Assumption~\ref{assumption:regression_function}, we can write:
\begin{align}\label{lemma:analogue-of-Th.5.2-2.5}
\norm{\thickbar{g}_{\theta}}_{\infty}&=\sup_{x_0,\ldots,x_{T-1}\in\sfX}\sqrt{\sumt\pars{\frac{1}{\sqrt{T}}\alpha g_t(x_t,\theta)}^2}\nonumber\\
&\leq\sup_{x_0,\ldots,x_{T-1}\in\sfX}\sqrt{\frac{1}{T}\sumt (f_t(x_t,\theta)-f_t(x_t,\theta_{\star}))^2}\nonumber\\
&\leq \sqrt{\frac{1}{T}\sumt L_1^2\norm{\theta-\theta_{\star}}^2}=L_1\norm{\theta-\theta_{\star}}\leq 2L_1\Btheta.
\end{align}
Moreover, we have:
\begin{align}\label{lemma:analogue-of-Th.5.2-2.25}
    \P\pars{\frac{1}{T}\sumt \thickbar{g}_t^2(X_t,\theta)\leq\frac{1}{2T}\sumt\E \thickbar{g}_t^2(X_t,\theta)}&=\P\pars{\frac{1}{T}\sumt \alpha^2 g_t^2(X_t,\theta)\leq\frac{1}{2T}\sumt\E \bracks{\alpha^2g_t^2(X_t,\theta)}}\nonumber\\
    &=\P\pars{\frac{1}{T}\sumt  g_t^2(X_t,\theta)\leq\frac{1}{2T}\sumt\E g_t^2(X_t,\theta)}\nonumber\\
    &\leq \exp\pars{-\frac{T}{8L_1^4a^2\normGamma^2}},
\end{align} 
where the last step follows from \eqref{lemma:analogue-of-Th.5.2-2}. From \eqref{lemma:analogue-of-Th.5.2-2.5},  \cite[Lemma 4.5]{Mendelson2002} and Lemma~\ref{lemma:G_covering_number} we have:
\begin{align}\label{lemma:analogue-of-Th.5.2-3}
    \abs{\calN_r}&=\calN\pars{\frac{r}{\sqrt{8}},\partial\calG_{\star,r},\norm{\cdot}_{\infty}}\nonumber\\
    &\leq\calN\pars{\frac{r}{\sqrt{8}},\starG,\norm{\cdot}_{\infty}}\hspace{3.5cm}(\partial\calG_{\star,r}\subseteq\starG)\nonumber\\
    &\leq \frac{4\sqrt{8}\sup_{\theta\in\sfM}\norm{\thickbar{g}_{\theta}}}{r}\calN\pars{\frac{r}{2\sqrt{8}},\calG,\norm{\cdot}_{\infty}}\hspace{0.6cm}(\text{from \cite[Lemma 4.5]{Mendelson2002}})\nonumber\\
    &\leq \frac{8\sqrt{8}L_1\Btheta}{r}\calN\pars{\frac{r}{2\sqrt{8}},\calG,\norm{\cdot}_{\infty}}\hspace{1.8cm}(\text{from \eqref{lemma:analogue-of-Th.5.2-2.5}})\nonumber\\
    &\leq \frac{8\sqrt{8}L_1\Btheta}{r}\pars{\frac{6\sqrt{8}L_1\Btheta}{r}}^{d_{\theta}}\hspace{2.5cm}(\text{from \eqref{G_covering_number}})\nonumber\\
    &\leq \pars{\frac{8\sqrt{8}L_1\Btheta}{r}}^{2d_{\theta}}.\hspace{4.1cm}(\dtheta\in\setN_+)
\end{align}
Note that every element of $\calN_r$ is of the form  $\alpha_ig_{\theta_i}$, where $\alpha_i\in[0,1]$ and $g_{\theta_i}\in\calG$. Define the event:
\begin{align*}  \calE=\bigcup_{\alpha_ig_{\theta_i}\in\calN_r}\left\{\frac{1}{T}\sumt \alpha_i^2g_t^2(X_t,\theta_i)\leq\frac{1}{2T}\sumt\E\bracks{\alpha_i^2g_t^2(X_t,\theta_i)}\right\}
\end{align*}
and observe, by a union bound, that:
\begin{align}\label{lemma:analogue-of-Th.5.2-4}
    \P(\calE)&\leq \sum_{\alpha_ig_{\theta_i}\in\calN_r}\P\pars{\frac{1}{T}\sumt \alpha_i^2g_t^2(X_t,\theta_i)\leq\frac{1}{2T}\sumt\E \bracks{\alpha_i^2g_t^2(X_t,\theta_i)}}\nonumber\\   &\leq\abs{\calN_r}\max_{\alpha_ig_{\theta_i}\in\calN_r}\P\pars{\frac{1}{T}\sumt \alpha_i^2g_t^2(X_t,\theta_i)\leq\frac{1}{2T}\sumt\E \bracks{\alpha_i^2g_t^2(X_t,\theta_i)}}.
\end{align}
Combining \eqref{lemma:analogue-of-Th.5.2-4} with \eqref{lemma:analogue-of-Th.5.2-2.25} and \eqref{lemma:analogue-of-Th.5.2-3}, and invoking Assumption~\ref{assumption:input_dependency}, we obtain:
\begin{align}\label{lemma:analogue-of-Th.5.2-5}
    \P(\calE)\leq \pars{\frac{8\sqrt{8}L_1\Btheta}{r}}^{2d_{\theta}}\exp\pars{-\frac{T^{1-b_2}}{8b_1L_1^4a^2}}.
\end{align}
Fix any $\thickbar{g}_{\theta}\in\partial\calG_{\star,r}$. By definition of $\calN_r$, there exists $\alpha_ig_{\theta_i}\in\calN_r$ such that $\norm{\thickbar{g}_{\theta}-\alpha_ig_{\theta_i}}_{\infty}\leq r/\sqrt{8}$. Hence, by Lemma~\ref{lemma:parallelogram}, we have:
\begin{align}\label{lemma:analogue-of-Th.5.2-6}
    \frac{1}{T}\sumt \thickbar{g}_t^2(X_t,\theta)&\geq\frac{1}{2T}\sumt \alpha_i^2g_t^2(X_t,\theta_i)-\frac{1}{T}\sumt(\thickbar{g}_t(X_t,\theta)-\alpha_i g_t(X_t,\theta_i))^2\nonumber\\
    &\geq\frac{1}{2T}\sumt \alpha_i^2g_t^2(X_t,\theta_i)-\norm{\thickbar{g}_{\theta}-\alpha_ig_{\theta_i}}_{\infty}^2\nonumber\\
    &\geq\frac{1}{2T}\sumt \alpha_i^2g_t^2(X_t,\theta_i)-\frac{r^2}{8}.
\end{align}
On the complement of $\calE$, \eqref{lemma:analogue-of-Th.5.2-6} yields:
\begin{align}\label{lemma:analogue-of-Th.5.2-7}
    &\frac{1}{T}\sumt \thickbar{g}_t^2(X_t,\theta)\geq\frac{1}{4T}\sumt\E\bracks{\alpha_i^2g_t^2(X_t,\theta_i)}-\frac{r^2}{8}=\frac{r^2}{4}-\frac{r^2}{8}=\frac{r^2}{8},
\end{align}
where the second-to-last step follows from the fact that $\alpha_ig_{\theta_i}\in\calN_r\subseteq\partial\calG_{\star,r}$. From \eqref{lemma:analogue-of-Th.5.2-5} and \eqref{lemma:analogue-of-Th.5.2-7} we conclude that:
\begin{align}\label{lemma:analogue-of-Th.5.2-8}
    &\P\pars{\inf_{\thickbar{g}_{\theta}\in\partial\calG_{\star,r}}\pars{\frac{1}{T}\sumt \thickbar{g}_t^2(X_t,\theta)-\frac{r^2}{8}}\leq0}\leq \pars{\frac{8\sqrt{8}L_1\Btheta}{r}}^{2d_{\theta}}\exp\pars{-\frac{T^{1-b_2}}{8b_1L_1^4a^2}}.
\end{align}
For every $\thickbar{g}_{\theta}\in\thickbar{\calG}_{\star,r}:=\starG\setminus\calG_{\star,r}$, there exists $r'>r$ such that:
\begin{align}\label{lemma:analogue-of-Th.5.2-8.5}
    &\frac{1}{T}\sumt \E \thickbar{g}_t^2(X_t,\theta)=r'^2 \implies\frac{1}{T}\sumt \E \pars{\frac{r}{r'}\thickbar{g}_t(X_t,\theta)}^2=r^2.
\end{align}
Since $\frac{r}{r'}<1$, by definition of $\starG$, we have $\frac{r}{r'}\thickbar{g}_{\theta}\in\starG$, and from \eqref{lemma:analogue-of-Th.5.2-8.5} we deduce that $\frac{r}{r'}\thickbar{g}_{\theta}\in\partial\calG_{\star,r}$.
Therefore, from \eqref{lemma:analogue-of-Th.5.2-8} we obtain:
\begin{align}\label{lemma:analogue-of-Th.5.2-9}
&\Prob\pars{\inf_{\thickbar{g}_{\theta}\in\thickbar{\calG}_{\star,r}}\pars{\frac{1}{T}\sumt \thickbar{g}_t^2(X_t,\theta)-\frac{1}{8T}\sumt \E \thickbar{g}_t^2(X_t,\theta)}\leq0}\leq\pars{\frac{8\sqrt{8}L_1\Btheta}{r}}^{2d_{\theta}}\exp\pars{-\frac{T^{1-b_2}}{8b_1L_1^4a^2}}.
\end{align}
Given that $\Big\{g_{\theta}\in\calG\,\big|\, \frac{1}{T}\sumt\E g_t^2(X_t,\theta)> r^2\Big\}\subseteq\thickbar{\calG}_{\star,r}$, \eqref{lemma:analogue-of-Th.5.2-9} implies \eqref{inequality:analogue-of-Th.5.2}, thus completing the proof.
\end{proof}

Next, we employ Lemma~\ref{lemma:analogue-of-Th.5.2} to prove the following result, which is a modified version of \cite[Theorem 4.1]{Ziemann2022} for \textit{time-varying} parametric regression functions $f_t(\cdot,\theta)$.

\begin{lemma}\label{lemma:analogue-of-Th.4.1}
Consider the predictor model \eqref{model} and the parameter class $\sfM$ under Assumptions~\cref{assumption:input_dependency,assumption:regression_function,assumption:expansivity}.
%~\cref{assumption:regression_function,assumption:expansivity,assumption:input_dependency}
Fix any $r\in(0,\sqrt{8}L_1\Btheta]$ and let $\hattheta$ satisfy \eqref{theta_hat}. Then, we have:
\begin{align}\label{inequality:analogue-of-Th.4.1}
    &\E_{\barX_{0:T-1}}\bracks{\frac{1}{T}\sumt(f_t(\barX_t,\hattheta)-f_t(\barX_t,\theta_{\star}))^2}\nonumber\\
    &\leq 8 M_T(\hat{\theta})+r^2+4L_1^2\Btheta^2\left(\frac{8\sqrt{8}L_1\Btheta}{r}\right)^{2\dtheta}\exp\pars{-\frac{T^{1-b_2}}{8b_1L_1^4a^2}},
\end{align}
where $\barX_{0:T-1}=(\barX_0,\ldots,\barX_{T-1})$.
\end{lemma}

\begin{proof}
Let $\calB_r$ and $g_t(\cdot,\cdot)$ be as in Lemma~\ref{lemma:analogue-of-Th.5.2} and define the event:
\begin{align*}
    \calE=&\left\{\inf_{\theta\in\sfM\setminus \calB_r}\left(\frac{1}{T}\sumt g_t^2(X_t,\theta)-\frac{1}{8T}\sumt \E g_t^2(X_t.\theta)\right)\leq0\right\}.
\end{align*}
Moreover, set $X_{0:T-1}=(X_0,\ldots,X_{T-1})$ and let:
\begin{equation*}
    N_T(X_{0:T-1},\theta)=\frac{1}{T}\sumt g_t^2(X_t,\theta),
\end{equation*}
for all input samples $X_{0:T-1}$ and parameters $\theta\in\sfM$. On the complement $\calE^{c}$ of $\calE$ we have:
\begin{align}\label{lemma:B.1-6}
\frac{1}{T}\sumt\E g_t^2(X_t,\theta)&\leq \max\left\{8N_T(X_{0:T-1},\theta),r^2\right\}\leq 8N_T(X_{0:T-1},\theta)+r^2,
\end{align}
for all input samples $X_{0:T-1}$ and parameters $\theta\in\sfM$. Given Assumption~\ref{assumption:regression_function} and the fact that  $\sfM\subseteq\setB_{B_{\theta}}^{d_{\theta}}$, we have:
\begin{align}\label{lemma:B.1-3.5}
    &\frac{1}{T}\sumt g_t^2(X_t,\theta)=\frac{1}{T}\sumt (f_t(X_t,\theta)-f_t(X_t,\theta_{\star}))^2\leq L_1^2\norm{\theta-\theta_{\star}}^2\leq4L_1^2\Btheta^2,
\end{align}
for all input samples $X_{0:T-1}$ and parameters $\theta\in\sfM$. Employing \eqref{lemma:B.1-6} and \eqref{lemma:B.1-3.5}, we can write: 
\begin{align}\label{lemma:B.1-7}
    \frac{1}{T}\sumt\E g_t^2(X_t,\theta) &= \frac{1}{T}\sumt\E\mathds{1}_{\calE^c}g_t^2(X_t,\theta)+\frac{1}{T}\sumt\E\mathds{1}_{\calE}g_t^2(X_t,\theta)\nonumber\\
    &\leq 8 N_T(X_{0:T-1},\theta)+r^2+4L_1^2\Btheta^2\P(\calE),
\end{align}
for all input samples $X_{0:T-1}$ and parameters $\theta\in\sfM$, given that  $\E\mathds{1}_{\calE}=\P(\calE)$. As a byproduct of the proof of \cite[Lemma 7]{Ziemann2022a}, we have the basic inequality:
\begin{align}\label{lemma:B.1-9.5}
    &N_T(X_{0:T-1},\hattheta)\leq M_T(\hattheta).
\end{align}
Setting $\theta=\hattheta$ in \eqref{lemma:B.1-7} and employing \eqref{inequality:analogue-of-Th.5.2} from Lemma~\ref{lemma:analogue-of-Th.5.2} and \eqref{lemma:B.1-9.5}, we obtain:
\begin{align*}
    &\E_{\barX_{0:T-1}} \bracks{\frac{1}{T}\sumt g_t^2(\barX_t,\hattheta)}\leq 8 M_T(\hattheta)+r^2+4L_1^2\Btheta^2\pars{\frac{8\sqrt{8}L_1\Btheta}{r}}^{2d_{\theta}}\exp\pars{-\frac{T^{1-b_2}}{8b_1L_1^4a^2}},
\end{align*}
where $\thickbar{X}_{0:T-1}=(\thickbar{X}_0,\ldots,\thickbar{X}_{T-1})$,
thus completing the proof.
\end{proof}

Now, we can prove Theorem~\ref{theorem:theorem2}. Fix any $\gamma\in[0,1)$ and set $r=\frac{L_1\Btheta}{T^{(1+\gamma)/2}}$. Then, if the third term on the right-hand side of \eqref{inequality:analogue-of-Th.4.1} is dominated by $r^2$, inequality \eqref{inequality:theorem2} trivially follows. We can write this term as follows:
\begin{align*}
4L_1^2\Btheta^2\pars{\frac{8\sqrt{8}L_1\Btheta}{r}}^{2d_{\theta}}\exp\pars{-\frac{T^{1-b_2}}{8b_1L_1^4a^2}}
&=4L_1^2\Btheta^2\exp\pars{2d_{\theta}\log\pars{\frac{8\sqrt{8}L_1\Btheta}{r}}-\frac{T^{1-b_2}}{8b_1L_1^4a^2}}\\
&=4L_1^2\Btheta^2\exp\pars{2d_{\theta}\log\pars{8\sqrt{8}T^{\frac{1+\gamma}{2}}}-\frac{T^{1-b_2}}{8b_1L_1^4a^2}}.
\end{align*}
Now, we choose $T$ large enough so that:
\begin{align*}
&2d_{\theta}\log\pars{8\sqrt{8}T^{\frac{1+\gamma}{2}}}-\frac{T^{1-b_2}}{8b_1L_1^4a^2}\leq-\frac{T^{1-b_2}}{16b_1L_1^4a^2}\\
\iff &\,T^{1-b_2}\geq 32d_{\theta}b_1L_1^4a^2\log\pars{8\sqrt{8}T^{\frac{1+\gamma}{2}}}\\
\iff &\,T^{1-b_2}\geq C\bigg(\log(8\sqrt{8})+\frac{1+\gamma}{2(1-b_2)}\log T^{1-b_2}\bigg),
\end{align*}
where $C=32d_{\theta}b_1L_1^4a^2$. Therefore, it suffices to require that:
\begin{equation}\label{theorem2-1}
    T^{1-b_2}\geq C\log\pars{8\sqrt{8}},\, T^{1-b_2}\geq\frac{C(1+\gamma)}{2(1-b_2)}\log T^{1-b_2}.
\end{equation}
By \cite[Lemma A.4]{Simchowitz2018}, the right inequality holds when:
\begin{equation*}
    T^{1-b_2}\geq \frac{C(1+\gamma)}{1-b_2}\log\pars{\frac{2C(1+\gamma)}{1-b_2}}.
\end{equation*}
If all the above requirements on $T$ hold, in order for $r^2$ to dominate the third term on the right-hand side of \eqref{inequality:analogue-of-Th.4.1}, it suffices to have:
\begin{align*}
&4L_1^2\Btheta^2\exp\pars{-\frac{2d_{\theta}T^{1-b_2}}{C}}\leq\frac{L_1^2\Btheta^2}{T^{1+\gamma}}\\
\iff &\,T^{1-b_2}\geq \frac{C}{2d_{\theta}}\log (4T^{1+\gamma})\\
\iff &\,T^{1-b_2}\geq\frac{C}{2\dtheta}\pars{\log4+\frac{1+\gamma}{1-b_2}\log T^{1-b_2}}.
\end{align*}
Hence, it suffices to require that:
\begin{equation*}
    T^{1-b_2}\geq \frac{C\log4}{2\dtheta},\, T^{1-b_2}\geq\frac{C(1+\gamma)}{2\dtheta(1-b_2)}\log T^{1-b_2},
\end{equation*}
which is ensured by \eqref{theorem2-1}, given that $\dtheta\in\setN_+$ and $\log(8\sqrt{8})\geq\log4/2$. Combining all of our conditions on $T$ with the fact that $\gamma<1$ (by assumption), we require that $T\geq T_1$, where:
\begin{align}\label{burn_in_time_1_exact}
    T_1= &\max\left\{\pars{32d_{\theta}b_1L_1^4a^2\log\pars{8\sqrt{8}}}^{1/(1-b_2)},\pars{\frac{64d_{\theta}b_1L_1^4a^2}{1-b_2}\log\pars{\frac{128d_{\theta}b_1L_1^4a^2}{1-b_2}}}^{1/(1-b_2)}\right\},   
\end{align}
which completes the proof.

\section{Proof of Lemma~\ref{lemma:taylor_bound}}\label{Proof_of_Lemma_1}

Employing the Taylor expansion given in \eqref{lemma:taylor_bound-2}, the martingale offset of $\hattheta$ can be written as follows:
\begin{align}\label{lemma:taylor_bound-4}
    &M_T(\hattheta)= \frac{1}{T}\sumt \left(4W_tg_t(X_t,\hattheta)-g_t^2(X_t,\hattheta)\right)\nonumber\\
    &=\frac{1}{T}\sum_{t=0}^{T-1}\left[4W_t\bigg(Z_t^{\T}(\hat{\theta}-\theta_{\star})+\frac{1}{2}(\hat{\theta}-\theta_{\star})^{\intercal}V_t(\hat{\theta}-\theta_{\star})\bigg)-\left(Z_t^{\T}(\hat{\theta}-\theta_{\star})+\frac{1}{2}(\hat{\theta}-\theta_{\star})^{\intercal}V_t(\hat{\theta}-\theta_{\star})\right)^2\right].
\end{align}
Invoking \eqref{parallelogram} from Lemma~\ref{lemma:parallelogram} and Cauchy-Schwarz inequality, \eqref{lemma:taylor_bound-4} yields:
\begin{align}\label{lemma:taylor_bound-5}
    M_T(\hattheta)\leq &\;\frac{1}{T}\sum_{t=0}^{T-1}\Bigg[4W_tZ_t^{\T}(\hat{\theta}-\theta_{\star})+2W_t(\hat{\theta}-\theta_{\star})^{\intercal}V_t(\hat{\theta}-\theta_{\star})\nonumber\\
    &\hspace{1.3cm}-\frac{1}{2}\big(Z_t^{\T}(\hat{\theta}-\theta_{\star})\big)^2+\bigg(\frac{1}{2}(\hat{\theta}-\theta_{\star})^{\intercal}V_t(\hat{\theta}-\theta_{\star})\bigg)^2\Bigg]\nonumber\\
    \leq &\;\thickbar{M}_T(\hattheta)+\normbig{\frac{2}{T}\sumt W_tV_t}\norm{\hat{\theta}-\theta_{\star}}^2+\frac{1}{4T}\sumt\big((\hat{\theta}-\theta_{\star})^{\intercal}V_t(\hat{\theta}-\theta_{\star})\big)^2,
\end{align}
where $\thickbar{M}_T(\hattheta)$ is defined as in \eqref{linearized_term}. For each $t$, we have:
\begin{align}
    \label{lemma:taylor_bound-8}
    &|(\hattheta-\theta_{\star})^{\T}V_t(\hattheta-\theta_{\star})|\leq \norm{V_t}\norm{\hattheta-\theta_{\star}}^2\leq L_2\norm{\hattheta-\theta_{\star}}^2,
\end{align}
where the last inequality follows from Assumption~\ref{assumption:regression_function}. Combining \eqref{lemma:taylor_bound-5} and \eqref{lemma:taylor_bound-8}, we obtain:
\begin{align*}
    M_T(\hattheta)\leq &\;\thickbar{M}_T(\hattheta)+\normbig{\frac{2}{T}\sumt W_tV_t}\norm{\hat{\theta}-\theta_{\star}}^2+\frac{L_2^2}{4}\norm{\hat{\theta}-\theta_{\star}}^4,
\end{align*}
which completes the proof.

\section{Proof of Theorem~\ref{theorem:expected_self_normalized_martingale}}\label{Proof_of_Theorem_3}

Let us define:
\begin{equation}\label{hatSigmaT}
    \hatSigmaT = \frac{1}{T}\sumt Z_tZ_t^{\T}
\end{equation}
and:
\begin{equation}\label{barSigma}
    \barSigma=\frac{1}{T}\sumt \E(Z_tZ_t^{\T}),
\end{equation}
where $Z_t=\gradtheta f_t(X_t,\theta_{\star})$, as in the Taylor expansion given in \eqref{lemma:taylor_bound-2}, for all $t=0,\ldots,T-1$.

First, we want to derive high-probability inequalities of the form:
\begin{equation*}
    C_1\barSigma\preceq\hatSigmaT\preceq C_2\barSigma,
\end{equation*}
where $C_1$ and $C_2$ are positive constants. We typically refer to the right inequality as \textit{upper isometry} and the left inequality as \textit{lower isometry}. We present our results in the next two lemmas, which rely on the concentration inequalities \eqref{inequality:analogue-of-Samson-Th.2-1} and \eqref{inequality:analogue-of-Samson-Th.2-2} from Lemma~\ref{lemma:analogue_of_Th.5.1}. 

\begin{lemma}[Upper Isometry]\label{lemma:upper_isometry}
Consider the predictor model \eqref{model} and the parameter class $\sfM$ under Assumptions~\cref{assumption:input_dependency,assumption:regression_function,assumption:excitation}.
Let $\hatSigmaT$ and $\barSigma$ be as in \eqref{hatSigmaT} and \eqref{barSigma}, respectively. Then, if $T\geq T_{21}$, where:
\begin{equation}\label{burn_in_time_upper}
    T_{21}=\pars{\frac{24d_{\theta}b_1L_1^2}{\lambda_0}\log\pars{\frac{6L_1}{\sqrt{\lambda_0}}}}^{1/(1-b_2)},
\end{equation}
we have:
\begin{equation*}
    \P\pars{\hatSigmaT\not\preceq8\barSigma}\leq\exp\pars{-\frac{\lambda_0T^{1-b_2}}{24b_1L_1^2}}.
\end{equation*}
\end{lemma}

\begin{proof}
Fix any $u\in\setS^{d_{\theta}-1}$ and for each $t$, set $g_t(X_t,\theta_{\star})=(u^{\T}Z_t)^2/T$. By Cauchy-Schwarz inequality and Assumption~\ref{assumption:regression_function}, notice that $g_t(X_t,\theta_{\star})\leq L_1^2/T$, for all $t=0,\ldots,T-1$. Employing  \eqref{inequality:analogue-of-Samson-Th.2-1} from Lemma~\ref{lemma:analogue_of_Th.5.1} with $\theta=\theta_{\star}$ and $s=\frac{1}{2}u^{\T}\barSigma u$, we obtain:
\begin{align}\label{lemma:upper_isometry-0.5}
&\P\pars{\frac{1}{T}\sumt(u^{\T}Z_t)^2\geq\frac{3}{2}(u^{\T}\barSigma^{1/2})^2}\leq\exp\pars{-\frac{(u^{\T}\barSigma u)T}{12L_1^2\normGamma^2}}.
\end{align}
Invoking Assumption~\ref{assumption:input_dependency}, \eqref{lemma:upper_isometry-0.5} yields:
\begin{align}\label{lemma:upper_isometry-1}
&\P\pars{\frac{1}{T}\sumt(u^{\T}Z_t)^2\geq\frac{3}{2}(u^{\T}\barSigma^{1/2})^2}\leq\exp\pars{-\frac{(u^{\T}\barSigma u)T^{1-b_2}}{12b_1L_1^2}}.
\end{align}
Fix any $\e\in(0,1]$ and let $\calN_{\e}$ be an $\e$-net of $\setS^{d_{\theta}-1}$ with respect to the norm $\norm{\cdot}$. Moreover, define the event:
\begin{align*}
\calE=\bigcup_{u_i\in\calN_{\e}}\left\{\frac{1}{T}\sumt(u_i^{\T}Z_t)^2\geq\frac{3}{2}(u_i^{\T}\barSigma^{1/2})^2\right\}.
\end{align*}
By a union bound and \eqref{lemma:upper_isometry-1}, we can write:
\begin{align}\label{lemma:upper_isometry-2.5}
\P(\calE)&\leq\sum_{u_i\in\calN_{\e}}\P\pars{\frac{1}{T}\sumt(u_i^{\T}Z_t)^2\geq\frac{3}{2}(u_i^{\T}\barSigma^{1/2})^2}\nonumber\\
&\leq\sum_{u_i\in\calN_{\e}}\exp\pars{-\frac{(u_i^{\T}\barSigma u_i)T^{1-b_2}}{12b_1L_1^2}}\nonumber\\
&\leq |\calN_{\e}|\max_{u_i\in\calN_{\e}}\exp\pars{-\frac{(u_i^{\T}\barSigma u_i)T^{1-b_2}}{12b_1L_1^2}}\nonumber\\
&\leq |\calN_{\e}|\sup_{u\in\setS^{\dtheta-1}}\exp\pars{-\frac{(u^{\T}\barSigma u)T^{1-b_2}}{12b_1L_1^2}}\nonumber\\
&= |\calN_{\e}|\exp\pars{-\frac{\inf_{u\in\setS^{\dtheta-1}}(u^{\T}\barSigma u)T^{1-b_2}}{12b_1L_1^2}}.
\end{align}
Combining \eqref{lemma:upper_isometry-2.5} with \eqref{covering_number} from Lemma~\ref{lemma:covering_number} and Assumption~\ref{assumption:excitation}, we get:
\begin{align}\label{lemma:upper_isometry-3}
\P(\calE)&\leq \pars{\frac{3}{\e}}^{d_{\theta}}\exp\pars{-\frac{\lambda_0T^{1-b_2}}{12b_1L_1^2}}.
\end{align}
By definition of $\calN_{\e}$, for every $u\in\setS^{d_{\theta}-1}$ there exists $u_i\in\calN_{\e}$ such that $\norm{u-u_i}\leq\e$, and thus on the complement $\calE^c$ of $\calE$ we have:
\begin{align}\label{lemma:upper_isometry-4}
    \frac{1}{T}\sumt(u^{\T}Z_t)^2&\leq\frac{2}{T}\sumt(u_i^{\T}Z_t)^2+\frac{2}{T}\sumt((u-u_i)^{\T}Z_t)^2\hspace{3.7cm}(\text{from }\eqref{parallelogram})\nonumber\\
    &\leq 3(u_i^{\T}\barSigma^{1/2})^2+\frac{2}{T}\sumt((u-u_i)^{\T}Z_t)^2 \hspace{4.6cm}(\calE^c)\nonumber\\
    &\leq 6(u^{\T}\barSigma^{1/2})^2+6((u-u_i)^{\T}\barSigma^{1/2})^2+\frac{2}{T}\sumt((u-u_i)^{\T}Z_t)^2 \hspace{0.8cm}(\text{from }\eqref{parallelogram})\nonumber\\
    &\leq 6(u^{\T}\barSigma^{1/2})^2+8\e^2L_1^2,
\end{align}
where the last step follows from Cauchy-Schwarz inequality, Assumption~\ref{assumption:regression_function}, and $\norm{u-u_i}\leq\e$. By Assumptions~\cref{assumption:regression_function,assumption:excitation}, we have $0<\lambda_0\leq u^{\T}\barSigma u=(u^{\T}\barSigma^{1/2})^2\leq L_1^2$. Hence, setting $\e=\frac{\sqrt{\lambda_0}}{2L_1}$, we have $\e\in(0,1]$ and \eqref{lemma:upper_isometry-4} yields:
\begin{align}\label{lemma:upper_isometry-5}
   &\frac{1}{T}\sumt(u^{\T}Z_t)^2\leq 6(u^{\T}\barSigma^{1/2})^2+2\lambda_0\leq8(u^{\T}\barSigma^{1/2})^2 \iff u^{\T}\hatSigmaT u\leq 8u^{\T}\barSigma u.
\end{align}
From \eqref{lemma:upper_isometry-3} and \eqref{lemma:upper_isometry-5} we get:
\begin{align*}
&\P\pars{\exists u\in\setS^{d_{\theta}-1}:u^{\T}(\hatSigmaT-8\barSigma) u\geq0}\leq \pars{\frac{6L_1}{\sqrt{\lambda_0}}}^{d_{\theta}}\exp\pars{-\frac{\lambda_0T^{1-b_2}}{12b_1L_1^2}},
\end{align*}
which implies that:
\begin{equation*}
    \P\pars{\hatSigmaT\not\preceq8\barSigma}\leq\pars{\frac{6L_1}{\sqrt{\lambda_0}}}^{d_{\theta}}\exp\pars{-\frac{\lambda_0T^{1-b_2}}{12b_1L_1^2}}.
\end{equation*}
To finish the proof, it suffices to show that if $T\geq T_{21}$, where $T_{21}$ is defined as in \eqref{burn_in_time_upper}, we have:
\begin{align*}
    &\pars{\frac{6L_1}{\sqrt{\lambda_0}}}^{d_{\theta}}\exp\pars{-\frac{\lambda_0T^{1-b_2}}{12b_1L_1^2}}\leq\exp\pars{-\frac{\lambda_0T^{1-b_2}}{24b_1L_1^2}}\\
    \iff &\,\exp\pars{d_{\theta}\log\pars{\frac{6L_1}{\sqrt{\lambda_0}}}-\frac{\lambda_0T^{1-b_2}}{12b_1L_1^2}}\leq\exp\pars{-\frac{\lambda_0T^{1-b_2}}{24b_1L_1^2}}\\
    \iff &\,T^{1-b_2}\geq \frac{24d_{\theta}b_1L_1^2}{\lambda_0}\log\pars{\frac{6L_1}{\sqrt{\lambda_0}}},
\end{align*}
which is true by the last inequality.
\end{proof}

\begin{lemma}[Lower Isometry]\label{lemma:lower_isometry}
Consider the predictor model \eqref{model} and the parameter class $\sfM$ under Assumptions~\cref{assumption:input_dependency,assumption:regression_function,assumption:excitation}. Let $\hatSigmaT$ and $\barSigma$ be as in \eqref{hatSigmaT} and \eqref{barSigma}, respectively. Then, if $T\geq T_{22}$, where:
\begin{equation}\label{burn_in_time_lower}
    T_{22}=\pars{\frac{16d_{\theta}b_1L_1^2}{\lambda_0}\log\pars{\frac{15L_1}{\sqrt{\lambda_0}}}}^{1/(1-b_2)},
\end{equation}
we have:
\begin{equation*}
    \P\pars{\hatSigmaT\not\succeq\frac{1}{16}\barSigma}\leq\exp\pars{-\frac{\lambda_0T^{1-b_2}}{16b_1L_1^2}}.
\end{equation*}
\end{lemma}

\begin{proof}
Fix any $u\in\setS^{d_{\theta}-1}$ and for each $t$, set $g_t(X_t,\theta_{\star})=(u^{\T}Z_t)^2/T$. By Cauchy-Schwarz inequality and Assumption~\ref{assumption:regression_function}, notice that $g_t(X_t,\theta_{\star})\leq L_1^2/T$, for all $t=0,\ldots,T-1$. Employing \eqref{inequality:analogue-of-Samson-Th.2-2} from Lemma~\ref{lemma:analogue_of_Th.5.1} with $\theta=\theta_{\star}$ and $s=\frac{1}{2}u^{\T}\barSigma u$, we obtain:
\begin{align}\label{lemma:lower_isometry-0.5}
&\P\pars{\frac{1}{T}\sumt(u^{\T}Z_t)^2\leq\frac{1}{2}(u^{\T}\barSigma^{1/2})^2}\leq\exp\pars{-\frac{(u^{\T}\barSigma u)T}{8L_1^2\normGamma^2}}.
\end{align}
Invoking Assumption~\ref{assumption:input_dependency}, \eqref{lemma:lower_isometry-0.5} yields:
\begin{align}\label{lemma:lower_isometry-1}
&\P\pars{\frac{1}{T}\sumt(u^{\T}Z_t)^2\leq\frac{1}{2}(u^{\T}\barSigma^{1/2})^2}\leq\exp\pars{-\frac{(u^{\T}\barSigma u)T^{1-b_2}}{8b_1L_1^2}}.
\end{align}
Fix any $\e\in(0,1]$ and let $\calN_{\e}$ be an $\e$-net of $\setS^{d_{\theta}-1}$ with respect to the norm $\norm{\cdot}$. Moreover, define the event:
\begin{align*}
\calE=\bigcup_{u_i\in\calN_{\e}}\left\{\frac{1}{T}\sumt(u_i^{\T}Z_t)^2\leq\frac{1}{2}(u_i^{\T}\barSigma^{1/2})^2\right\}.
\end{align*}
By a union bound and \eqref{lemma:lower_isometry-1}, we can write:
\begin{align}\label{lemma:lower_isometry-2.5}
\P(\calE)&\leq\sum_{u_i\in\calN_{\e}}\P\pars{\frac{1}{T}\sumt(u_i^{\T}Z_t)^2\leq\frac{1}{2}(u_i^{\T}\barSigma^{1/2})^2}\nonumber\\
&\leq\sum_{u_i\in\calN_{\e}}\exp\pars{-\frac{(u_i^{\T}\barSigma u_i)T^{1-b_2}}{8b_1L_1^2}}\nonumber\\
&\leq |\calN_{\e}|\max_{u_i\in\calN_{\e}}\exp\pars{-\frac{(u_i^{\T}\barSigma u_i)T^{1-b_2}}{8b_1L_1^2}}\nonumber\\
&\leq |\calN_{\e}|\sup_{u\in\setS^{\dtheta-1}}\exp\pars{-\frac{(u^{\T}\barSigma u)T^{1-b_2}}{8b_1L_1^2}}\nonumber\\
&= |\calN_{\e}|\exp\pars{-\frac{\inf_{u\in\setS^{\dtheta-1}}(u^{\T}\barSigma u)T^{1-b_2}}{8b_1L_1^2}}.
\end{align}
Combining \eqref{lemma:lower_isometry-2.5} with \eqref{covering_number} from Lemma~\ref{lemma:covering_number} and Assumption~\ref{assumption:excitation}, we get:
\begin{align}\label{lemma:lower_isometry-3}
\P(\calE)&\leq \pars{\frac{3}{\e}}^{d_{\theta}}\exp\pars{-\frac{\lambda_0T^{1-b_2}}{8b_1L_1^2}}.
\end{align}
By definition of $\calN_{\e}$, for every $u\in\setS^{d_{\theta}-1}$ there exists $u_i\in\calN_{\e}$ such that $\norm{u-u_i}\leq\e$, and thus on the complement $\calE^c$ of $\calE$ we have:
\begin{align}\label{lemma:lower_isometry-4}
    \frac{1}{T}\sumt(u^{\T}Z_t)^2&\geq\frac{1}{2T}\sumt(u_i^{\T}Z_t)^2-\frac{1}{T}\sumt((u-u_i)^{\T}Z_t)^2\hspace{4cm}(\text{from }\eqref{parallelogram})\nonumber\\
    &\geq \frac{1}{4}(u_i^{\T}\barSigma^{1/2})^2-\frac{1}{T}\sumt((u-u_i)^{\T}Z_t)^2 \hspace{5.1cm}(\calE^c)\nonumber\\
    &\geq \frac{1}{8}(u^{\T}\barSigma^{1/2})^2-\frac{1}{4}((u-u_i)^{\T}\barSigma^{1/2})^2-\frac{1}{T}\sumt((u-u_i)^{\T}Z_t)^2 \hspace{1.2cm}(\text{from }\eqref{parallelogram})\nonumber\\
    &\geq \frac{1}{8}(u^{\T}\barSigma^{1/2})^2-\frac{5}{4}\e^2L_1^2,
\end{align}
where the last step follows from Cauchy-Schwarz inequality, Assumption~\ref{assumption:regression_function}, and $\norm{u-u_i}\leq\e$. By Assumptions~\cref{assumption:regression_function,assumption:excitation}, we have $0<\lambda_0\leq u^{\T}\barSigma u=(u^{\T}\barSigma^{1/2})^2\leq L_1^2$. Hence, setting $\e=\frac{\sqrt{\lambda_0}}{\sqrt{20}L_1}$, we have $\e\in(0,1]$ and \eqref{lemma:lower_isometry-4} yields:
\begin{align}\label{lemma:lower_isometry-5}
   &\frac{1}{T}\sumt(u^{\T}Z_t)^2\geq \frac{1}{8}(u^{\T}\barSigma^{1/2})^2-\frac{\lambda_0}{16}\geq\frac{1}{16}(u^{\T}\barSigma^{1/2})^2 \iff u^{T}\hatSigmaT u\geq \frac{1}{16}u^{\T}\barSigma u.
\end{align}
From \eqref{lemma:lower_isometry-3} and \eqref{lemma:lower_isometry-5} we get:
\begin{align*}
&\P\pars{\exists u\in\setS^{d_{\theta}-1}:u^{\T}\pars{\hatSigmaT-\frac{1}{16}\barSigma} u\leq0}\leq \pars{\frac{15L_1}{\sqrt{\lambda_0}}}^{d_{\theta}}\exp\pars{-\frac{\lambda_0T^{1-b_2}}{8b_1L_1^2}},
\end{align*}
which implies that:
\begin{equation*}
    \P\pars{\hatSigmaT\not\succeq\frac{1}{16}\barSigma}\leq\pars{\frac{15L_1}{\sqrt{\lambda_0}}}^{d_{\theta}}\exp\pars{-\frac{\lambda_0T^{1-b_2}}{8b_1L_1^2}}.
\end{equation*}
To finish the proof, it suffices to show that if $T\geq T_{22}$, where $T_{22}$ is defined as in \eqref{burn_in_time_lower}, we have:
\begin{align*}
    &\pars{\frac{15L_1}{\sqrt{\lambda_0}}}^{d_{\theta}}\exp\pars{-\frac{\lambda_0T^{1-b_2}}{8b_1L_1^2}}\leq\exp\pars{-\frac{\lambda_0T^{1-b_2}}{16b_1L_1^2}}\\
    \iff &\exp\pars{d_{\theta}\log\pars{\frac{15L_1}{\sqrt{\lambda_0}}}-\frac{\lambda_0T^{1-b_2}}{8b_1L_1^2}}\leq\exp\pars{-\frac{\lambda_0T^{1-b_2}}{16b_1L_1^2}}\\
    \iff &T^{1-b_2}\geq \frac{16d_{\theta}b_1L_1^2}{\lambda_0}\log\pars{\frac{15L_1}{\sqrt{\lambda_0}}},
\end{align*}
which is true by the last inequality.
\end{proof}

Now, we can prove Theorem~\ref{theorem:expected_self_normalized_martingale}. We start by showing an upper bound of order $\calO(\log T/T)$ for $\E[\sup_{\theta\in\sfM}\thickbar{M}_T(\theta)]$, where $\thickbar{M}_T(\theta)$ is defined as $\thickbar{M}_T(\hattheta)$ in \eqref{linearized_term} with $\hattheta$ replaced by $\theta$, for all $\theta\in\sfM$.
 
Fix any $\e>0$ and let $\calN_{\e}$ be an $\e$-net of $\sfM_{\star}:=\{\theta-\theta_{\star}\where \theta\in\sfM\}$. Similar to \cite[Lemma 10]{Ziemann2022a}, we can show that:
\begin{align}\label{lemma:martingale_offset_complexity_bound-1.1}
&\E\bracks{\sup_{\theta\in\sfM}\thickbar{M}_T(\theta)}\leq\frac{16\sigma_w^2\log\calN(\e,\sfM,\norm{\cdot})}{T}+32\sigma_wL_1\e.
\end{align}
Particularly, note that invoking Assumption~\ref{assumption:noise}, similar to \cite[Lemma 9]{Ziemann2022a} we have:
\begin{align}\label{lemma:martingale_offset_complexity_bound-1.2}
\E\bracks{\max_{\theta-\theta_{\star}\in\calN_{\e}}\thickbar{M}_T(\theta)}\leq\frac{16\sigma_w^2\log|\calN_{\e}|}{T}.
\end{align}
Moreover, observe that by simply discarding the negative second-order term in $\thickbar{M}_T(\theta)$ and employing Assumption~\ref{assumption:regression_function} along with a standard sub-Gaussian concentration inequality for $\E|W_t|$, we get:
\begin{align}\label{lemma:martingale_offset_complexity_bound-1.3}
    \E\bracks{\sup_{\substack{\theta\in\sfM\\\norm{\theta-\theta_{\star}}\leq\e}}\thickbar{M}_T(\theta)}&\leq \E\bracks{\sup_{\substack{\theta\in\sfM\\\norm{\theta-\theta_{\star}}\leq\e}}\frac{4}{T}\sumt W_t g_t(X_t,\theta)}\nonumber\\
    &\leq \E\bracks{\sup_{\substack{\theta\in\sfM\\\norm{\theta-\theta_{\star}}\leq\e}}\frac{4}{T}\sumt \abs{W_t} \abs{f_t(X_t,\theta)-f_t(X_t,\theta_{\star})}}\nonumber\\
    &\leq \E\bracks{\sup_{\substack{\theta\in\sfM\\\norm{\theta-\theta_{\star}}\leq\e}}\frac{4}{T}\sumt \abs{W_t} L_1\norm{\theta-\theta_{\star}}}\nonumber\\
    &\leq 16\sigma_w L_1\e.
\end{align}
A standard one-step discretization bound (c.f. the proof of\cite[Proposition 5.17]{Wainwright2019}) yields:
\begin{align}\label{lemma:martingale_offset_complexity_bound-1.4}
&\E\bracks{\sup_{\theta\in\sfM}\thickbar{M}_T(\theta)}\leq \E\bracks{\max_{\theta-\theta_{\star}\in\calN_{\e}}\thickbar{M}_T(\theta)}+2\E\bracks{\sup_{\substack{\theta\in\sfM\\\norm{\theta-\theta_{\star}}\leq\e}}\thickbar{M}_T(\theta)}.
\end{align}
Combining \eqref{lemma:martingale_offset_complexity_bound-1.4} with \eqref{lemma:martingale_offset_complexity_bound-1.2} and \eqref{lemma:martingale_offset_complexity_bound-1.3}, and noting that $\calN(\e,\sfM,\norm{\cdot})=|\calN_{\e}|$ (by translation invariance of the metric $\norm{\cdot}$), we complete the proof of \eqref{lemma:martingale_offset_complexity_bound-1.1}.

Setting $\e=\frac{\Btheta}{T}$ in \eqref{lemma:martingale_offset_complexity_bound-1.1} and invoking \eqref{M_covering_number} from Lemma~\ref{lemma:M_covering_number}, we obtain:
\begin{align}\label{lemma:martingale_offset_complexity_bound-2}
&\E\bracks{\sup_{\theta\in\sfM}\thickbar{M}_T(\theta)}\leq\frac{16\sigma_w^2d_{\theta}\log(3T)}{T}+\frac{32\sigma_wL_1\Btheta}{T}.
\end{align}
Given that $\log(3T)\geq\log(3)>1$, for all $T\in\setN_+$, \eqref{lemma:martingale_offset_complexity_bound-2} yields:
\begin{align}\label{theorem:self_normalized_martingale-0.5}
\E\bracks{\sup_{\theta\in\sfM}\thickbar{M}_{T}(\theta)}\leq \frac{16(\sigma_w^2d_{\theta}+2\sigma_wL_1\Btheta)\log (3T)}{T}.
\end{align}
We have:
\begin{align}\label{theorem:self_normalized_martingale-1}
&\E\thickbar{M}_T(\hattheta)\leq\sup_{\theta\in\sfM}\E \thickbar{M}_T(\theta)\leq\E\bracks{\sup_{\theta\in\sfM}\thickbar{M}_{T}(\theta)}\leq \frac{16(\sigma_w^2d_{\theta}+2\sigma_wL_1\Btheta)\log(3T)}{T},
\end{align}
where the last inequality follows from \eqref{theorem:self_normalized_martingale-0.5}. Let us define the event $\calE=\calE_1\cup\calE_2$, where:
\begin{align*}
    \calE_1=\left\{\hatSigmaT\not\preceq8\barSigma\right\},\, \calE_2=\left\{\hatSigmaT\not\succeq\frac{1}{16}\barSigma\right\}.
\end{align*}
By a union bound and Lemmas~\ref{lemma:upper_isometry} and~\ref{lemma:lower_isometry}, if $T\geq\max\{T_{21},T_{22}\}$, where $T_{21}$, $T_{22}$ are given by \eqref{burn_in_time_upper} and \eqref{burn_in_time_lower}, respectively, we have:
\begin{align}\label{theorem:self_normalized_martingale-1.5}
    &\P(\calE)\leq \P(\calE_1)+\P(\calE_2)\leq\exp\pars{-\frac{\lambda_0T^{1-b_2}}{24b_1L_1^2}}+\exp\pars{-\frac{\lambda_0T^{1-b_2}}{16b_1L_1^2}}\leq2\exp\pars{-\frac{\lambda_0T^{1-b_2}}{24b_1L_1^2}}.
\end{align}
By Assumption~\ref{assumption:excitation}, on the complement $\calE^c$ of $\calE$ we have: 
\begin{align}\label{theorem:self_normalized_martingale-1.6}
    \hatSigmaT\succeq\frac{1}{16}\barSigma\succeq\frac{\lambda_0}{16}\mathbb{I}_{\dtheta}\succ0. 
\end{align}
Hence, by maximizing the quadratic function $\thickbar{M}_T(\theta)$ over all $\theta\in\setR^{\dtheta}$ on $\calE^c$ and performing straightforward algebraic manipulations, we obtain:
\begin{align}\label{theorem:self_normalized_martingale-2}
    \thickbar{M}_T(\hattheta)&\leq\sup_{\theta\in\setR^{\dtheta}}\thickbar{M}_T(\theta)\leq \frac{8}{T}\normbig{\sumt W_tZ_t^{\T}(T\hatSigmaT)^{-1/2}}^2= \frac{16}{T}\normbig{\sumt W_tZ_t^{\T}(2T\hatSigmaT)^{-1/2}}^2\nonumber\\
    &\leq \frac{16}{T}\normbig{\sumt W_tZ_t^{\T}(\Sigma+T\hatSigmaT)^{-1/2}}^2,
\end{align}
where the last inequality follows from \eqref{theorem:self_normalized_martingale-1.6} for $\Sigma=\frac{T}{16}\barSigma$. Fix any $\delta\in(0,1)$. Employing \cite[Theorem 4.1]{Ziemann2023} (by Assumption~\ref{assumption:noise}) and \eqref{theorem:self_normalized_martingale-2}, we deduce that with probability at least $1-\delta$, on $\calE^c$ we have:
\begin{align}\label{theorem:self_normalized_martingale-3}
\thickbar{M}_T(\hattheta)&\leq\frac{16}{T}\bigg(4\sigma_w^2\log\pars{\frac{\det(\Sigma+T\hatSigmaT)}{\det(\Sigma)}}+8\sigma_w^2\log\pars{\frac{5}{\delta}}\bigg)\nonumber\\
&=\frac{16}{T}\bigg(4\sigma_w^2\log\det(\mathbb{I}_{\dtheta}+T\hatSigmaT\Sigma^{-1})+8\sigma_w^2\log\pars{\frac{5}{\delta}}\bigg)\nonumber\\
&\leq\frac{16}{T}\bigg(4\sigma_w^2\log\det\pars{\mathbb{I}_{\dtheta}+T(8\barSigma)\pars{\frac{16}{T}\barSigma^{-1}}}+8\sigma_w^2\log\pars{\frac{5}{\delta}}\bigg)  \hspace{2cm}( \calE_1^c )\nonumber\\
&=\frac{16}{T}\bigg(4\sigma_w^2\log129^{d_{\theta}}+8\sigma_w^2\log\pars{\frac{5}{\delta}}\bigg)\nonumber\\
&\leq\frac{128\sigma_w^2}{T}\log\pars{\frac{5\cdot12^{d_{\theta}}}{\delta}}.
\end{align}
For any $s>\frac{128\sigma_w^2}{T}\log(5\cdot12^{d_{\theta}})$, we can select $\delta\in(0,1)$ such that $s=\frac{128\sigma_w^2}{T}\log\pars{\frac{5\cdot12^{d_{\theta}}}{\delta}}$, and  \eqref{theorem:self_normalized_martingale-3} yields:
\begin{align*}
    \P\pars{\thickbar{M}_T(\hattheta)>s\,\big|\,\calE^c}\leq5\cdot12^{d_{\theta}}\exp\pars{-\frac{Ts}{128\sigma_w^2}}.
\end{align*}
Define the event $\tilde{\calE}=\big\{\thickbar{M}_T(\hattheta)<0\big\}$. We have:
\begin{align*}
    &\E\bracks{\thickbar{M}_T(\hattheta)\,\big|\, \calE^c}= \E\bracks{\mathds{1}_{\tilde{\calE}^c}\thickbar{M}_T(\hattheta)\,\big|\, \calE^c}+\E\bracks{\mathds{1}_{\tilde{\calE}}\thickbar{M}_T(\hattheta)\,\big|\, \calE^c}\leq \E\bracks{\mathds{1}_{\tilde{\calE}^c}\thickbar{M}_T(\hattheta)\,\big|\, \calE^c},
\end{align*}
where the last step follows from the fact that $\E\bracks{\mathds{1}_{\tilde{\calE}}\thickbar{M}_T(\hattheta)\,\big|\, \calE^c}<0$, by definition of $\tilde{\calE}$. Therefore, we have:
\begin{align}\label{theorem:self_normalized_martingale-4}
    \E \bracks{\thickbar{M}_T(\hattheta)\,\big|\,\calE^c}&\leq \E\bracks{\mathds{1}_{\tilde{\calE}^c}\thickbar{M}_T(\hattheta)\,\big|\, \calE^c}=\int_{0}^{\infty}\P\pars{\thickbar{M}_T(\hattheta)>s\where\calE^c}ds\nonumber\\
    &\leq\int_0^{\frac{128\sigma_w^2}{T}\log(5\cdot12^{d_{\theta}})}1ds+ \int_{\frac{128\sigma_w^2}{T}\log(5\cdot12^{d_{\theta}})}^{\infty}5\cdot12^{d_{\theta}}\exp\pars{-\frac{Ts}{128\sigma_w^2}}ds\nonumber\\
    &= \frac{128\sigma_w^2\log(5\cdot12^{\dtheta})}{T}+\int_0^{\infty}\exp\pars{-\frac{Ts}{128\sigma_w^2}}ds\nonumber\\
    &=\frac{128\sigma_w^2(\log5+\dtheta\log12)}{T}+\frac{128\sigma_w^2}{T}\nonumber\\
    &\leq\frac{128\sigma_w^2\dtheta\log60}{T}+\frac{128\sigma_w^2\dtheta}{T}\hspace{3.3cm}(\dtheta\in\setN_+)\nonumber\\
    &\leq \frac{653d_{\theta}\sigma_w^2}{T}.
\end{align}
Invoking the law of total expectation along with \eqref{theorem:self_normalized_martingale-1} and \eqref{theorem:self_normalized_martingale-4}, we can write:
\begin{align}\label{theorem:self_normalized_martingale-5}
    \E \thickbar{M}_T(\hattheta)    &=\E\bracks{\thickbar{M}_T(\hattheta)\,\big|\,\calE^c}\P(\calE^c)+\E \bracks{\thickbar{M}_T(\hattheta)\,\big|\,\calE}\P(\calE)\nonumber\\
    &\leq \frac{653d_{\theta}\sigma_w^2}{T}\P(\calE^c)+\frac{16(\sigma_w^2d_{\theta}+2\sigma_wL_1\Btheta)\log(3T)}{T}\P(\calE).
\end{align}
Employing the fact that $\P(\calE^c)\leq1$ as well as \eqref{theorem:self_normalized_martingale-1.5}, \eqref{theorem:self_normalized_martingale-5} yields:
\begin{align}\label{theorem:self_normalized_martingale-6}
 \E \thickbar{M}_T(\hattheta)&\leq \frac{653d_{\theta}\sigma_w^2}{T}+\frac{C\log(3T)}{3T}\exp\pars{-\frac{\lambda_0T^{1-b_2}}{24b_1L_1^2}},   
\end{align}
where $C=96(\sigma_w^2d_{\theta}+2\sigma_wL_1\Btheta)$. Now we choose $T$ large enough so that:
\begin{align*}
    &C\exp\pars{-\frac{\lambda_0T^{1-b_2}}{24b_1L_1^2}}\leq\exp\pars{-\frac{\lambda_0T^{1-b_2}}{48b_1L_1^2}}\nonumber\\
    \iff &\,\exp\pars{\log C-\frac{\lambda_0T^{1-b_2}}{24b_1L_1^2}}\leq \exp\pars{-\frac{\lambda_0T^{1-b_2}}{48b_1L_1^2}}\nonumber\\
    \iff &\,T^{1-b_2}\geq\frac{48b_1L_1^2}{\lambda_0}\log C.
\end{align*}
Moreover, for any fixed $\gamma\in(0,1)$, we require that:
\begin{align*}
    &\exp\pars{-\frac{\lambda_0T^{1-b_2}}{48b_1L_1^2}}\leq \frac{1}{T^{1+\gamma}}\iff T^{1-b_2}\geq \frac{48b_1L_1^2}{\lambda_0}\log T^{1+\gamma}\iff T^{1-b_2}\geq \frac{48b_1L_1^2(1+\gamma)}{\lambda_0(1-b_2)}\log T^{1-b_2}.
\end{align*}
By \cite[Lemma A.4]{Simchowitz2018}, the latter inequality holds when:
\begin{align*}
    T^{1-b_2}\geq \frac{96b_1L_1^2(1+\gamma)}{\lambda_0(1-b_2)}\log\pars{\frac{192b_1L_1^2(1+\gamma)}{\lambda_0(1-b_2)}},
\end{align*}
and given that $\gamma<1$, it suffices to have:
\begin{align*}
    T^{1-b_2}\geq \frac{192b_1L_1^2}{\lambda_0(1-b_2)}\log\pars{\frac{384b_1L_1^2}{\lambda_0(1-b_2)}}.
\end{align*}
Combining all of our requirements on $T$, we need to have $T\geq\max\{T_{21},T_{22},T_{23}\}$, where:
\begin{align}\label{burn_in_T23}
    &T_{23}= \max\left\{
    \pars{\frac{192b_1L_1^2}{\lambda_0(1-b_2)}\log\pars{\frac{384b_1L_1^2}{\lambda_0(1-b_2)}}}^{\frac{1}{1-b_2}},\pars{\frac{48b_1L_1^2\log (96(\sigma_w^2d_{\theta}+2\sigma_wL_1\Btheta))}{\lambda_0}}^{\frac{1}{1-b_2}}\right\}.\nonumber\\
\end{align}
That is, we set:
\begin{align}\label{burn_in_time_2_exact}
    T_2=\max\{T_{21},T_{22},T_{23}\},
\end{align}
where $T_{21}$, $T_{22}$, $T_{23}$ are given by \eqref{burn_in_time_upper}, \eqref{burn_in_time_lower}, and \eqref{burn_in_T23}, respectively.
Then, given that $\log(3T)\leq3T$, for all $T\in\setN_+$, \eqref{theorem:self_normalized_martingale-5} yields:
\begin{align*}
    \E \thickbar{M}_T(\hattheta)\leq\frac{653d_{\theta}\sigma_w^2}{T}+\frac{1}{T^{1+\gamma}},
\end{align*}
which completes the proof.

\section{Proof of Corollary~\ref{corollary:expected_M__hattheta_bound}}\label{Proof_of_Corollary_expected_M_hattheta_bound}
  
We want to find an upper bound for $\E M_T(\hattheta)$, employing the inequality:
\begin{align}\label{corollary:corollary2-1}
    \E M_T(\hattheta)\leq &\;\E\thickbar{M}_T(\hattheta)+\E\bracks{\normbig{\frac{2}{T}\sumt W_tV_t}\norm{\hat{\theta}-\theta_{\star}}^2}+\frac{L_2^2}{4}\E\norm{\hat{\theta}-\theta_{\star}}^4,
\end{align}
which follows directly from Lemma~\ref{lemma:taylor_bound} (see \eqref{theorem:theorem2-2}). The first term on the right-hand side of \eqref{corollary:corollary2-1} can be bounded using \eqref{expected_self_normalized_martingale} from Theorem~\ref{theorem:expected_self_normalized_martingale}. 

We proceed with bounding the second term on the right-hand side of \eqref{corollary:corollary2-1}. For every $\theta\in\sfM$, let $M_T(\theta)$ be defined as $M_T(\hattheta)$ with $\hattheta$ replaced by $\theta$. Moreover, set $\sfM_{\star}=\{\theta-\theta_{\star}\where\theta\in\sfM\}$. Invoking Assumption~\ref{assumption:regression_function}, similar to \cite[Theorem 4.2.2]{Ziemann2022thesis}, we can show that for any fixed $u,v,w\geq0$, with probability at least $1-3\exp(-\frac{u^2}{2})-\exp(-\frac{v}{2})-e\exp(-w)$, we have:
\begin{align*}
    \sup_{\theta\in\sfM}M_T(\theta)\leq\inf_{\alpha>0,\,\delta\in[0,\alpha]} 16\bigg(&wL_1\delta\sigma_w+\sqrt{\frac{\sigma_w^2}{T}}\int_{\delta/2}^{\alpha}\sqrt{\log\calN(\e,\sfM_{\star},\norm{\cdot})}d\e+\frac{v\sigma_w^2}{T}\nonumber\\
    &+\frac{\sigma_w^2\log\calN(\alpha,\sfM_{\star},\norm{\cdot})}{T}+\frac{uL_1\alpha\sigma_w}{\sqrt{T}}+L_1^2\alpha^2\bigg).
\end{align*}
Note that for any $\e>0$, we have $\calN(\e,\sfM_{\star},\norm{\cdot})=\calN(\e,\sfM,\norm{\cdot})$, by translation invariance of the metric $\norm{\cdot}$. Hence, by selecting $\alpha=\frac{\Btheta}{T}$, $\delta=\frac{2\Btheta}{T}$, $u=2\sqrt{\log(3T)}$, $v=2\log(3T^2)$, and $w=1+\log(3T^2)$, we conclude that with probability at least $1-\frac{1}{T^2}$:
\begin{align}\label{corollary:corollary2-2}
    &\sup_{\theta\in\sfM}M_T(\theta)\leq 16\pars{w L_1\delta\sigma_w+\frac{v\sigma_w^2}{T}+\frac{\sigma_w^2\log\calN(\Btheta/T,\sfM,\norm{\cdot})}{T}+\frac{u L_1\Btheta\sigma_w}{T^{3/2}}+\frac{L_1^2\Btheta^2}{T^2}}.
\end{align}
Combining \eqref{corollary:corollary2-2} with \eqref{M_covering_number} from Lemma~\ref{lemma:M_covering_number} and performing straightforward algebraic manipulations, we obtain:
\begin{align}\label{corollary:corollary2-3}
    \P\pars{\sup_{\theta\in\sfM}M_T(\theta)> r^2}\leq\frac{1}{T^2},
\end{align}
where $r^2=C_1\frac{\log(3T)}{T}$ with:
\begin{align*}
C_1=16(d_{\theta}\sigma_w^2+4\sigma_w^2+10\sigma_w L_1\Btheta+L_1^2\Btheta^2).
\end{align*}
Define the event:
\begin{align*}
    \calE=\left\{\norm{\hattheta-\theta_{\star}}>\thickbar{r}\right\},
\end{align*}
where:
\begin{align*}
    \thickbar{r}=\sqrt{8ar^2+\frac{2aL_1^2\Btheta^2}{T}}.
\end{align*}
If $T\geq T_1$, where $T_1$ is defined as in \eqref{burn_in_time_1_exact}, from \eqref{inequality:theorem2_old} and \eqref{corollary:corollary2-3} we can write:
\begin{align}\label{corollary:corollary2-4}
    \P(\calE)&=\P\pars{\norm{\hattheta-\theta_{\star}}^2>\thickbar{r}^2}\nonumber\\    &\leq\P\pars{8aM_T(\hattheta)+\frac{2aL_1^2\Btheta^2}{T}>8ar^2+\frac{2aL_1^2\Btheta^2}{T}}\nonumber\\
    &=\P\pars{M_T(\hattheta)> r^2}\nonumber\\
    &\leq\P\pars{\sup_{\theta\in\sfM}M_T(\theta)> r^2}\leq\frac{1}{T^2}.
\end{align}
We can decompose the second term on the right-hand side of \eqref{corollary:corollary2-1} as follows:
\begin{align}\label{corollary:corollary2-5}
&\E\bracks{\normbig{\frac{2}{T}\sumt W_tV_t}\norm{\hat{\theta}-\theta_{\star}}^2}=E_1+E_2,
\end{align}
where:
\begin{equation*}
E_1=\E\bracks{\mathds{1}_{\calE^c}\normbig{\frac{2}{T}\sumt W_tV_t}\norm{\hat{\theta}-\theta_{\star}}^2}
\end{equation*}
and:
\begin{equation*}
E_2=\E\bracks{\mathds{1}_{\calE}\normbig{\frac{2}{T}\sumt W_tV_t}\norm{\hat{\theta}-\theta_{\star}}^2}.
\end{equation*}
First, we want to bound the term $E_1$. Let us define:
\begin{align*}
    \sfM_{\thickbar{r}}=\left\{\theta\in\sfM\,\Big|\,\norm{\theta-\theta_{\star}}\leq\thickbar{r}\right\},
\end{align*}
and (somewhat abusing notation):
\begin{align*}
    V_t(\theta)=\gradtheta^2 f_t(X_t,\alpha_t(\theta-\theta_{\star})+\theta_{\star}),
\end{align*}
where $\alpha_t$ are the same as in the definition of $V_t$ in \eqref{lemma:taylor_bound-2}, for all $\theta\in\sfM$. Then, we have:
\begin{align}\label{corollary:corollary2-6}
    E_1&\leq\E\bracks{\sup_{\theta\in\sfM_{\thickbar{r}}}\normbig{\frac{2}{T}\sumt W_tV_t(\theta)}\norm{\theta-\theta_{\star}}^2}\leq\thickbar{r}^2\,\E\bracks{\sup_{\theta\in\sfM_{\thickbar{r}}}\normbig{\frac{2}{T}\sumt W_tV_t(\theta)}}.
\end{align}
Fix any $\theta\in\sfM_{\thickbar{r}}$ and any $u\in\setS^{d_{\theta}-1}$, and define the random variable:
\begin{equation*}
    R_u(\theta)=u^{\T}\pars{\frac{2}{T}\sumt W_tV_t(\theta)}u.
\end{equation*}
By tower property and Assumptions~\ref{assumption:noise},~\ref{assumption:regression_function}, for any $\lambda\in\setR$, we can write:
\begin{align*}
\E\exp\pars{\lambda R_u(\theta)}&=\E\exp\pars{\frac{2\lambda}{T}\sumt W_t(u^{\T}V_t(\theta)u)}\nonumber\\
&=\E\bigg[\exp\bigg(\frac{2\lambda}{T}\sum_{t=0}^{T-2} W_t(u^{\T}V_t(\theta)u)\bigg)\E\bigg[\exp\pars{\frac{2\lambda}{T} W_{T-1}(u^{\T}V_{T-1}(\theta)u)}\,\Big|\,\calF_{T-2}\bigg]\bigg]\nonumber\\
&\leq \E\bigg[\exp\bigg(\frac{2\lambda}{T}\sum_{t=0}^{T-2} W_t(u^{\T}V_t(\theta)u)\bigg)\exp\pars{\frac{2\lambda^2\sigma_w^2L_2^2}{T^2}}\bigg]\nonumber\\
&\leq \ldots\nonumber\\
&\leq \E\exp\pars{\frac{2\lambda^2\sigma_w^2L_2^2}{T}},
\end{align*}
which implies that $R_u(\theta)$ is sub-Gaussian with parameter $\sigma=(2\sigma_wL_2)/\sqrt{T}$. From standard results for sub-Gaussian random variables we deduce that for any $s>0$:
\begin{align}\label{corollary:corollary2-7}
    \P\pars{\abs{R_u(\theta)}> s}\leq2\exp\pars{-\frac{Ts^2}{8\sigma_w^2L_2^2}}. 
\end{align}
Fix any $\e\in(0,1/2)$ and let $\calN_{\e}$ be an $\e$-net of $\setS^{d_{\theta}-1}$. Then, from \cite[Lemma 2.5]{Ziemann2023} and \eqref{corollary:corollary2-7} we conclude that:
\begin{align}\label{corollary:corollary2-8}
    \P\pars{\normbig{\frac{2}{T}\sumt W_tV_t(\theta)}> s}
    &\leq \pars{\frac{3}{\e}}^{d_{\theta}}\max_{u\in\calN_{\e}}\P\pars{\abs{R_u(\theta)}>(1-2\e)s}\nonumber\\
    &\leq 2\pars{\frac{3}{\e}}^{d_{\theta}}\exp\pars{-\frac{T(1-2\e)^2s^2}{8\sigma_w^2L_2^2}}.
\end{align}
Setting $\e=1/4$, \eqref{corollary:corollary2-8} yields:
\begin{align}\label{corollary:corollary2-8.5}
\P\pars{\normbig{\frac{2}{T}\sumt W_tV_t(\theta)}>s}&\leq2\cdot12^{d_{\theta}}\exp\pars{-\frac{Ts^2}{32\sigma_w^2L_2^2}}\nonumber\\
&\leq24^{d_{\theta}}\exp\pars{-\frac{Ts^2}{32\sigma_w^2L_2^2}},\hspace{0.7cm}(\dtheta\in\setN_+)
\end{align}
for all $\theta\in\sfM$. Let $\calN_{\thickbar{r}}'$ be a $\thickbar{r}$-net of $\sfM_{\thickbar{r},\star}:=\{\theta-\theta_{\star}\where\theta\in\sfM_{\thickbar{r}}\}$. Then, by a union bound, \eqref{corollary:corollary2-8.5} yields:
\begin{align}\label{corollary:corollary2-8.75}
&\P\pars{\max_{\theta_i\in\calN_{\thickbar{r}}'}\normbig{\frac{2}{T}\sumt W_tV_t(\theta_i+\theta_{\star})}>s}\leq|\calN_{\thickbar{r}}'|\,24^{d_{\theta}}\exp\pars{-\frac{Ts^2}{32\sigma_w^2L_2^2}}.
\end{align}
Similar to \eqref{M_covering_number} from Lemma~\ref{lemma:M_covering_number}, we can show that:
\begin{align*}
    |\calN_{\thickbar{r}}'|=\calN(\thickbar{r},\sfM_{\thickbar{r},\star},\norm{\cdot})\leq\pars{\frac{3\thickbar{r}}{\thickbar{r}}}^{\dtheta}=3^{\dtheta}.
\end{align*}
Hence, from \eqref{corollary:corollary2-8.75} we deduce that:
\begin{align*}
&\P\pars{\max_{\theta_i\in\calN_{\thickbar{r}}'}\normbig{\frac{2}{T}\sumt W_tV_t(\theta_i+\theta_{\star})}>s}\leq72^{d_{\theta}}\exp\pars{-\frac{Ts^2}{32\sigma_w^2L_2^2}}
\end{align*}
and thus we have:
\begin{align}\label{corollary:corollary2-9}
\E\bracks{\max_{\theta_i\in\calN_{\thickbar{r}}'}\normbig{\frac{2}{T}\sumt W_tV_t(\theta_i+\theta_{\star})}}
&=\int_0^{\infty}\P\pars{\max_{\theta_i\in\calN_{\thickbar{r}}'}\normbig{\frac{2}{T}\sumt W_tV_t(\theta_i+\theta_{\star})}>s}ds\nonumber\\
&\leq\int_0^{\infty}\min\left\{1,72^{d_{\theta}}\exp\pars{-\frac{Ts^2}{32\sigma_w^2L_2^2}}\right\}ds\nonumber\\
&=\int_0^{\infty}\min\left\{1,\exp\pars{\log72^{\dtheta}-\frac{Ts^2}{32\sigma_w^2L_2^2}}\right\}ds\nonumber\\
&=\sqrt{\frac{32\sigma_w^2L_2^2\log 72^{d_{\theta}}}{T}}+\int_0^{\infty}\exp\pars{-\frac{Ts^2}{32\sigma_w^2L_2^2}}ds\nonumber\\
&=\sqrt{\frac{32\dtheta\sigma_w^2L_2^2\log 72}{T}}+\sqrt{\frac{8\pi\sigma_w^2L_2^2}{T}}\nonumber\\
&\leq \sqrt{\frac{\dtheta\sigma_w^2L_2^2}{T}}\pars{\sqrt{32\log72}+\sqrt{8\pi}} \hspace{1cm}(\dtheta\in\setN_+)\nonumber\\
&\leq \frac{18\sqrt{d_{\theta}}\sigma_wL_2}{\sqrt{T}}.
\end{align}
By definition of $\calN_{\thickbar{r}}'$, for every $\theta\in\sfM_{\thickbar{r}}$ there exists $\theta_i\in\calN_{\thickbar{r}}'$ such that $\norm{(\theta-\theta_{\star})-\theta_i}\leq \thickbar{r}$. We define the operator $\pi(\cdot)$ that projects each $\theta\in\sfM_{\thickbar{r}}$ on the corresponding vector $\theta_i+\theta_{\star}$. Then, for every $\theta\in\sfM_{\thickbar{r}}$, we have $\norm{\theta-\pi(\theta)}\leq \thickbar{r}$. Then, employing the triangle inequality along with Assumption~\ref{assumption:regression_function}, we can write:
\begin{align*}
  \E\bracks{\sup_{\theta\in\sfM_{\thickbar{r}}}\normbig{\frac{2}{T}\sumt W_tV_t(\theta)}}
  &\leq \E\Bigg[\sup_{\substack{\theta\in\sfM_{\thickbar{r}}}}\Bigg(\normbig{\frac{2}{T}\sumt W_t(V_t(\theta)-V_t(\pi(\theta)))}+\normbig{\frac{2}{T}\sumt W_tV_t(\pi(\theta))}\Bigg)\Bigg]\nonumber\\
  &\leq \E\Bigg[\sup_{\substack{\theta\in\sfM_{\thickbar{r}}}}\Bigg(\frac{2}{T}\sumt |W_t|\norm{V_t(\theta)-V_t(\pi(\theta))}+\normbig{\frac{2}{T}\sumt W_tV_t(\pi(\theta))}\Bigg)\Bigg]\nonumber\\
  &\leq \E\bracks{\frac{2}{T}\sumt \abs{W_t}}L_3\thickbar{r}+\E\bracks{\max_{\theta_i\in\calN_{\thickbar{r}}'}\normbig{\frac{2}{T}\sumt W_tV_t(\theta_i+\theta_{\star})}}.
\end{align*}
Invoking \eqref{corollary:corollary2-9} and a standard sub-Gaussian concentration inequality for $\E|W_t|$, the above inequality yields:
\begin{align}\label{corollary:corollary2-10} &\E\bracks{\sup_{\theta\in\sfM_{\thickbar{r}}}\normbig{\frac{2}{T}\sumt W_tV_t(\theta)}}\leq 8\sigma_wL_3\thickbar{r}+\frac{18\sqrt{d_{\theta}}\sigma_wL_2}{\sqrt{T}}.
\end{align}
Combining inequalities \eqref{corollary:corollary2-6} and \eqref{corollary:corollary2-10}, we obtain:
\begin{align}\label{corollary:corollary2-11}
E_1\leq\thickbar{r}^2\pars{8\sigma_wL_3\thickbar{r}+\frac{18\sqrt{d_{\theta}}\sigma_wL_2}{\sqrt{T}}}.
\end{align}
Given that $1<\log3\leq\log(3T)$, for all $T\in\setN_+$, by definition of $\thickbar{r}$ we have:
\begin{align}\label{corollary:corollary2-11.1}
    \thickbar{r}\leq\sqrt{\frac{(8aC_1+2aL_1^2\Btheta^2)\log(3T)}{T}}
\end{align}
and \eqref{corollary:corollary2-11} yields:
\begin{align}\label{corollary:corollary2-11.2}
    E_1&\leq \pars{\frac{(8aC_1+2aL_1^2\Btheta^2)\log(3T)}{T}}\pars{8\sigma_w L_3\sqrt{\frac{(8aC_1+2aL_1^2\Btheta^2)\log(3T)}{T}}+\frac{18\sqrt{d_{\theta}}\sigma_wL_2}{\sqrt{T}}}\nonumber\\
    &\leq \pars{\frac{(8aC_1+2aL_1^2\Btheta^2)\log(3T)}{T}}\pars{\pars{8\sigma_wL_3\sqrt{8aC_1+2aL_1^2\Btheta^2}+18\sqrt{\dtheta}\sigma_wL_2}\sqrt{\frac{\log(3T)}{T}}}\nonumber\\
    &=C_2\pars{\frac{\log(3T)}{T}}^{3/2},
\end{align}
where:
\begin{align*}
C_2=&\pars{8aC_1+2aL_1^2\Btheta^2}\pars{8\sigma_wL_3\sqrt{8aC_1+2aL_1^2\Btheta^2}+18\sqrt{\dtheta}\sigma_wL_2}.
\end{align*}
Now, we focus on bounding the term $E_2$ appearing on the right-hand side of \eqref{corollary:corollary2-5}. Invoking the Cauchy-Schwarz inequality, we can write:
\begin{align}\label{corollary:corollary2-12}
    E_2&\leq\sqrt{\E\bracks{\normbig{\frac{2}{T}\sumt W_tV_t}^2\norm{\hat{\theta}-\theta_{\star}}^4}\E\mathds{1}_{\calE}^2}\nonumber\\
&\leq\sqrt{\E\bracks{\sup_{\theta\in\sfM}\normbig{\frac{2}{T}\sumt W_tV_t(\theta)}^2\norm{\theta-\theta_{\star}}^4}\P(\calE)}\hspace{1cm}(\E\mathds{1}_{\calE}^2=\E\mathds{1}_{\calE}=\P(\calE))\nonumber\\
&\leq4\Btheta^2\sqrt{\E\bracks{\sup_{\theta\in\sfM}\normbig{\frac{2}{T}\sumt W_tV_t(\theta)}^2}\P(\calE)}.
\end{align}
Let $\calN''$ be a $\Btheta/\sqrt{T}$-net of $\sfM$. By definition of $\calN''$, for every $\theta\in\sfM$ there exists $\theta_i\in\calN''$ such that $\norm{\theta-\theta_i}\leq \Btheta/\sqrt{T}$. Somewhat abusing notation, we define $\pi(\cdot)$ as the operator that projects each $\theta\in\sfM$ on the corresponding vector $\theta_i\in\calN''$. Then, by a union bound, \eqref{corollary:corollary2-8.5} yields:
\begin{align}\label{corollary:corollary2-12.75}
&\P\pars{\max_{\theta_i\in\calN''}\normbig{\frac{2}{T}\sumt W_tV_t(\theta_i)}>s}\leq|\calN''|\,24^{d_{\theta}}\exp\pars{-\frac{Ts^2}{32\sigma_w^2L_2^2}}.
\end{align}
Employing \eqref{M_covering_number} from Lemma~\ref{lemma:M_covering_number}, we get:
\begin{align*}
    |\calN''|=\calN(\Btheta/\sqrt{T},\sfM,\norm{\cdot})\leq(3\sqrt{T})^{\dtheta},
\end{align*}
and \eqref{corollary:corollary2-12.75} yields:
\begin{align*}
&\P\pars{\max_{\theta_i\in\calN''}\normbig{\frac{2}{T}\sumt W_tV_t(\theta_i)}>s}\leq(72\sqrt{T})^{d_{\theta}}\exp\pars{-\frac{Ts^2}{32\sigma_w^2L_2^2}}.
\end{align*}
Therefore, we have:
\begin{align}\label{corollary:corollary2-13}
\E\bracks{\max_{\theta_i\in\calN''}\normbig{\frac{2}{T}\sumt W_tV_t(\theta_i)}^2}
&=\int_0^{\infty}\P\pars{\max_{\theta_i\in\calN''}\normbig{\frac{2}{T}\sumt W_tV_t(\theta_i)}^2>s}ds\nonumber\\
&=\int_0^{\infty}\P\pars{\max_{\theta_i\in\calN''}\normbig{\frac{2}{T}\sumt W_tV_t(\theta_i)}>\sqrt{s}}ds\nonumber\\
&\leq\int_0^{\infty}\min\left\{1,(72\sqrt{T})^{d_{\theta}}\exp\pars{-\frac{Ts}{32\sigma_w^2L_2^2}}\right\}ds\nonumber\\
&=\int_0^{\infty}\min\left\{1,\exp\pars{\log(72\sqrt{T})^{d_{\theta}}-\frac{Ts}{32\sigma_w^2L_2^2}}\right\}ds\nonumber\\
&=\frac{32\sigma_w^2L_2^2\log(72\sqrt{T})^{\dtheta}}{T}+\int_0^{\infty}\exp\pars{-\frac{Ts}{32\sigma_w^2L_2^2}}ds\nonumber\\
&= \frac{32d_{\theta}\sigma_w^2L_2^2\log(72\sqrt{T})}{T}+\frac{32\sigma_w^2L_2^2}{T}\nonumber\\
&\leq\frac{32\dtheta\sigma_w^2L_2^2}{T}\pars{\log(72\sqrt{T})+1}\hspace{1cm}(\dtheta\in\setN_+)\nonumber\\
&\leq \frac{176d_{\theta}\sigma_w^2L_2^2\log(3T)}{T},
\end{align}
where the last inequality follows by using the fact that $1<\log3\leq\log(3T)$, for all $T\in\setN_+$, and performing straightforward algebraic manipulations. Employing \eqref{parallelogram} from Lemma~\ref{lemma:parallelogram} along with Cauchy-Schwarz inequality and Assumption~\ref{assumption:regression_function}, we can write:
\begin{align*}
    \normbig{\frac{2}{T}\sumt W_tV_t(\theta)}^2&\leq 2\normbig{\frac{2}{T}\sumt W_t(V_t(\theta)-V_t(\pi(\theta)))}^2+2\normbig{\frac{2}{T}\sumt W_tV_t(\pi(\theta))}^2\nonumber\\
    &\leq \frac{8}{T^2}\pars{\sumt W_t^2}\pars{\sumt \norm{V_t(\theta)-V_t(\pi(\theta))}^2}+2\normbig{\frac{2}{T}\sumt W_tV_t(\pi(\theta))}^2\nonumber\\
    &\leq \pars{\frac{8}{T}\sumt W_t^2}\pars{L_3\norm{\theta-\pi(\theta)}}^2+2\normbig{\frac{2}{T}\sumt W_tV_t(\pi(\theta))}^2\nonumber\\
    &\leq \pars{\frac{8}{T}\sumt W_t^2}\pars{\frac{L_3\Btheta}{\sqrt{T}}}^2+2\normbig{\frac{2}{T}\sumt W_tV_t(\pi(\theta))}^2,
\end{align*}
for all $\theta\in\sfM$. Hence, we have:
\begin{align}\label{corollary:corollary2-13.5}
    \E\bracks{\sup_{\theta\in\sfM}\normbig{\frac{2}{T}\sumt W_tV_t(\theta)}^2}
    &\leq \E\bracks{\frac{8}{T}\sumt W_t^2}\pars{\frac{L_3\Btheta}{\sqrt{T}}}^2+2\E\bracks{\sup_{\theta\in\sfM}\normbig{\frac{2}{T}\sumt W_tV_t(\pi(\theta))}^2}\nonumber\\
  &= \E\bracks{\frac{8}{T}\sumt W_t^2}\pars{\frac{L_3\Btheta}{\sqrt{T}}}^2+2\E\bracks{\max_{\theta_i\in\calN''}\normbig{\frac{2}{T}\sumt W_tV_t(\theta_i)}^2}.
\end{align}
Invoking \eqref{corollary:corollary2-13} and a standard sub-Gaussian concentration inequality for $\E W_t^2$, \eqref{corollary:corollary2-13.5} yields:
\begin{align}\label{corollary:corollary2-14}
    \E\bracks{\sup_{\theta\in\sfM}\normbig{\frac{2}{T}\sumt W_tV_t(\theta)}^2}
  &\leq \frac{32\sigma_w^2L_3^2\Btheta^2}{T}+\frac{352d_{\theta}\sigma_w^2L_2^2\log(3T)}{T}\nonumber\\
  &\leq \frac{2d_{\theta}\sigma_w^2(16L_3^2\Btheta^2+176L_2^2)\log (3T)}{T},
\end{align}
where the last step follows from the fact that $\dtheta\in\setN_+$ and $1<\log3\leq\log(3T)$, for all $T\in\setN_+$. Combining \eqref{corollary:corollary2-12} with \eqref{corollary:corollary2-4} and \eqref{corollary:corollary2-14}, we obtain:
\begin{align}\label{corollary:corollary2-15}
    E_2\leq \frac{4\Btheta^2\sqrt{2d_{\theta}\sigma_w^2(16L_3^2\Btheta^2+176L_2^2)\log(3T)}}{T^{3/2}}.
\end{align}
Given that $1<\log3\leq\log(3T)\leq\log^3(3T)$, for all $T\in\setN_+$, \eqref{corollary:corollary2-15} yields:
\begin{align}\label{corollary:corollary2-15.1}
    E_2\leq C_3\pars{\frac{\log(3T)}{T}}^{3/2},
\end{align}
where:
\begin{align*}
C_3=4\Btheta^2\sqrt{2d_{\theta}\sigma_w^2(16L_3^2\Btheta^2+176L_2^2)}.
\end{align*}
Employing \eqref{corollary:corollary2-11.2} and \eqref{corollary:corollary2-15.1}, \eqref{corollary:corollary2-5} yields:
\begin{align}\label{corollary:corollary2-16}
   &\E\bracks{\normbig{\frac{2}{T}\sumt W_tV_t}\norm{\hat{\theta}-\theta_{\star}}^2}\leq(C_2+C_3)\pars{\frac{\log(3T)}{T}}^{3/2}.
\end{align}

Next, we derive a bound for the third term on the right-hand side of \eqref{corollary:corollary2-1}. Notice that  we can bound $\E\norm{\hattheta-\theta_{\star}}^4$ as follows:
\begin{align}\label{corollary:corollary2-20}
\E\norm{\hattheta-\theta_{\star}}^4&=\E\mathds{1}_{\calE^c}\norm{\hattheta-\theta_{\star}}^4+\E\mathds{1}_{\calE}\norm{\hattheta-\theta_{\star}}^4\nonumber\\
&\leq\pars{\sup_{\theta\in\sfM_{\thickbar{r}}}\norm{\theta-\theta_{\star}}^4}\P(\calE^c)+\pars{\sup_{\theta\in\sfM}\norm{\theta-\theta_{\star}}^4}\P(\calE)\hspace{0.4cm}(\E\mathds{1}_{\calE}=\P(\calE),\E\mathds{1}_{\calE^c}=\P(\calE^c))\nonumber\\
&\leq \thickbar{r}^4+\frac{16\Btheta^4}{T^2},
\end{align}
where the last inequality follows from the definition of $\sfM_{\thickbar{r}}$, the fact that $\P(\calE^c)\leq1$, and \eqref{corollary:corollary2-4}. From \eqref{corollary:corollary2-11.1} and \eqref{corollary:corollary2-20} we conclude that: 
\begin{align}\label{corollary:corollary2-21}
\E\norm{\hattheta-\theta_{\star}}^4&\leq \pars{\frac{(8aC_1+2aL_1^2\Btheta^2)\log(3T)}{T}}^2+\frac{16\Btheta^4}{T^2}.   
\end{align}
Given that $1<\log3\leq\log(3T)\leq3T$, for all $T\in\setN_+$, \eqref{corollary:corollary2-21} implies that the third term on the right-hand side of \eqref{corollary:corollary2-1} can be bounded as follows:
\begin{align}\label{corollary:corollary2-22}
\frac{L_2^2}{4}\E\norm{\hattheta-\theta_{\star}}^4
&\leq \frac{L_2^2}{4} \pars{(8aC_1+2aL_1^2\Btheta^2)^2+16\Btheta^4}\pars{\frac{\log(3T)}{T}}^2\nonumber\\
&=\frac{9L_2^2}{4} \pars{(8aC_1+2aL_1^2\Btheta^2)^2+16\Btheta^4}\pars{\frac{\log(3T)}{3T}}^2\nonumber\\
&\leq\frac{9L_2^2}{4} \pars{(8aC_1+2aL_1^2\Btheta^2)^2+16\Btheta^4}\pars{\frac{\log(3T)}{3T}}^{3/2}\nonumber\\
&\leq C_4\pars{\frac{\log(3T)}{T}}^{3/2},
\end{align}
where:
\begin{align*}
    C_4=L_2^2\pars{(8aC_1+2aL_1^2\Btheta^2)^2+16\Btheta^4}.
\end{align*}

Fix any $\gamma\in(0,1/2)$. Recall that the first term on the right-hand side of \eqref{corollary:corollary2-1} can be bounded using \eqref{expected_self_normalized_martingale} from Theorem~\ref{theorem:expected_self_normalized_martingale}, if $T\geq T_2$, where $T_2$ is defined as in \eqref{burn_in_time_2_exact}. From \eqref{corollary:corollary2-1}, \eqref{corollary:corollary2-16}, and \eqref{corollary:corollary2-22} we deduce that we can finish the proof of \eqref{expected_M__hattheta_bound} by choosing $T$ large enough so that:
\begin{align*}
  \max\{(C_2+C_3),C_4\}\pars{\frac{\log(3T)}{T}}^{3/2}\leq \frac{1}{2T^{1+\gamma}}.
\end{align*}
Setting $C_5=2\max\{(C_2+C_3),C_4\}$, it suffices to have:
\begin{align*}
    &T^{1/2-\gamma}\geq C_5\log^{3/2}(3T)\\
    \iff &T^{(1-2\gamma)/3}\geq C_5^{2/3}\log(3T)\\
    \iff &T^{(1-2\gamma)/3}\geq C_5^{2/3}\bigg(\log3+\frac{3}{1-2\gamma}\log T^{(1-2\gamma)/3}\bigg),
\end{align*}
which is satisfied if:
\begin{align*}
    T^{(1-2\gamma)/3}\geq C_5^{2/3}\log3,\; T^{(1-2\gamma)/3}\geq \frac{3C_5^{2/3}}{1-2\gamma}\log T^{(1-2\gamma)/3}.
\end{align*}
By \cite[Lemma A.4]{Simchowitz2018}, the latter inequality holds when:
\begin{align*}
    T^{(1-2\gamma)/3}\geq\frac{6C_5^{2/3}}{1-2\gamma}\log\pars{\frac{12C_5^{2/3}}{1-2\gamma}}.
\end{align*}
Combining all of our requirements on $T$, we need to have $T\geq\max\{T_1,T_2,T_3\}$, where $T_1$ and $T_2$ are given by \eqref{burn_in_time_1_exact} and \eqref{burn_in_time_2_exact}, respectively, and:
\begin{align*}
    T_3=&\max\left\{\pars{C_5^{2/3}\log3}^{3/(1-2\gamma)},\pars{\frac{6C_5^{2/3}}{1-2\gamma}\log\pars{\frac{12C_5^{2/3}}{1-2\gamma}}}^{3/(1-2\gamma)}\right\},
\end{align*}
which completes the proof.

\end{document}